\documentclass[12pt,twoside]{article}

\usepackage{amsmath, amssymb, amsthm, physics, url, bbm, mathrsfs, cite, xcolor, mathtools, dsfont, esint}
\usepackage[shortlabels]{enumitem}
\mathtoolsset{showonlyrefs}

\usepackage[pagewise]{lineno}

\usepackage[T1]{fontenc}

\definecolor{darkblue}{RGB}{32,57,231}
\usepackage[colorlinks,
	citecolor=darkblue, 
	linkcolor=darkblue, 
	urlcolor=darkblue,
	bookmarksopen=true,
	pdfauthor={Aobo Chen and Zhenyu Yu}, 
	pdftitle={Spectral bounds and heat kernel upper estimates for Dirichlet forms}
]{hyperref}

\usepackage[open,openlevel=1]{bookmark}

\newtheorem{theorem}{Theorem}[section]
\newtheorem{proposition}[theorem]{Proposition}
\newtheorem{lemma}[theorem]{Lemma}
\newtheorem{corollary}[theorem]{Corollary}

\theoremstyle{definition}
\newtheorem{definition}[theorem]{Definition}
\newtheorem{example}[theorem]{Example}
\newtheorem{remark}[theorem]{Remark}
\newtheorem{notation}[theorem]{Notation}

\numberwithin{equation}{section}

\DeclareMathOperator*{\esup}{esup}
\DeclareMathOperator*{\diam}{diam}

\newcommand{\one}{\mathds{1}}
\def\diag{\mathrm{diag}}
\newcommand*{\dif}{\mathop{}\!\mathrm{d}}

 \def\sE {{\mathcal E}} \def\sF {{\mathcal F}}

  \def\rL {{\mathscr L}}

 \def\bE {{\mathbb E}}

 \def\bN {{\mathbb N}} 
\def\bP {{\mathbb P}}  \def\bR {{\mathbb R}}

\newcommand\restr[2]{{
		\left.\kern-\nulldelimiterspace 
		#1 
		\vphantom{\big|}
		\right|_{#2}
}}

\date{\today}

\font\titlefont=cmbx12 scaled 1400
\title{\titlefont Spectral bounds and heat kernel upper estimates for Dirichlet forms}

\author{Aobo Chen and Zhenyu Yu}
	
\begin{document}

\maketitle

\vspace{-0.7cm}

\begin{abstract}
We use a Harnack-type inequality on exit times and spectral bounds to characterize upper bounds of the heat kernel associated with any regular Dirichlet form without killing part, where the scale function may vary with position. We further show that this Harnack-type inequality is preserved under quasi-symmetric changes of metric on uniformly perfect metric spaces. This {generalizes} the work of Mariano and Wang [Stochastic Process. Appl. 189 (2025) 104707].


\end{abstract}

\section{Introduction}

The fundamental solution of the heat equation (or heat kernel) $\partial_{t} u=\Delta u$ on $\mathbb{R}^{n}$ is {the} Gauss--Weierstrass kernel 
\begin{equation}
p_{t}(x,y)=\frac{1}{(4\pi t)^{n/2}}\exp\left(-\frac{\lvert x-y\rvert^{2} }{4t}\right).
\end{equation}
 For many Laplacians on fractals and metric measure {spaces} (such as uniformly elliptic, divergence-form {operators} on $\mathbb{R}^{n}$, the standard Laplacian on the Sierpi\'nski gasket \cite{BP88} and carpet \cite{BB92}), the following \emph{sub-Gaussian} upper bound {for the} heat kernel usually holds for some $\beta>1$:
\begin{equation}\label{UEloc}
p_{t}(x,y)\leq \frac{C}{\mu(B(x,t^{1/\beta}))}\exp\left(-c\left(\frac{d(x,y)^{\beta}}{t}\right)^{1/(\beta-1)}\right),
\end{equation}
where $d$ is the metric and $\mu(B(x,r))$ is the volume of the ball centered at $x$ with radius $r>0$. {There is a large literature} devoted to the study of sub-Gaussian heat kernel bounds on various spaces, for example, \cite{Gri09} on {Riemannian manifolds}, \cite{Bar98} on the Sierpi\'nski gasket, and \cite{AB15, GHL14,GHL15, BCS15, MSC17} on equivalent characterizations of heat kernel upper bounds for strongly local Dirichlet forms on general metric measure spaces. In particular, in \cite[Conjecture 3.9]{GHL14}, Grigor'yan, Hu and Lau conjecture that the heat kernel upper bound in \eqref{UEloc} is equivalent to the conjunction of {an} isoperimetric inequality (named \emph{Faber--Krahn inequality} \eqref{FK}) and the \emph{regular decay of {the bottom of the spectrum of} balls} \eqref{lambda} {for a regular, strongly local, conservative Dirichlet form when $\mu$ is Ahlfors regular}. Recently, Mariano and Wang \cite[Proposition 7.8]{MW25} proved that for unbounded spaces, \eqref{UEloc} is equivalent to the conjunction of these two conditions plus a \emph{Harnack-type inequality for exit times}.

In this paper, we are mainly interested in generalizing \cite[Proposition 7.8]{MW25} to non-local settings. When a {long-range} non-local term is involved, the sub-Gaussian bound
\eqref{UEloc} no longer holds. Instead, a \emph{stable-like} heat kernel upper bound should take its place:
\begin{equation}\label{e.UEpoly}
p_{t}(x,y)\leq C \left( \frac{1}{\mu(B(x, t^{1/\beta}))} \wedge \frac{t}{\mu(B(x,
d(x,y))) d(x,y)^{\beta}} \right)\end{equation}
for some positive constant $C$ and $\beta$. There {is also a large literature devoted to the study of such upper bounds}, for example \cite{BGK09} for jump processes, and \cite{HL18,CKW21, HL22, KW23, GHH24c} for regular Dirichlet forms containing non-local parts on doubling metric measure spaces.

It is natural to consider the non-local analogue of \cite[Conjecture 3.9]{GHL14}. In fact, we will work in a more general setting. For a regular, conservative Dirichlet form without a killing part on metric measure spaces with \emph{volume doubling property} and \emph{reverse volume doubling property}, we will prove in Theorem \ref{t.main} that the stable-like heat kernel upper bound \eqref{e.UEpoly} is equivalent to the conjunction of Faber--Krahn inequality \eqref{FK}, the decay of {the bottom of the spectrum} \eqref{lambda}, the Harnack-type inequality \eqref{EM} for exit times and an upper bound for jump kernels \eqref{J<}. This Harnack-type inequality is not easy to verify for a single space. Nevertheless, for uniformly perfect metric spaces, we prove that \eqref{EM} is \emph{invariant under quasi-symmetric transformation of metrics}. Since quasi-symmetry theory plays an important role \cite{Kig12, KM23}, we expect that the quasi-symmetric invariance of \eqref{EM} could also help in understanding heat kernel bounds in future works. 

Another interesting result in \cite{MW25} is to obtain tail estimates of survival probability at large times. Let $D$ be an open subset of a metric measure space {such that the bottom of the spectrum of $D$, denoted by $\lambda(D)$}, is positive. For Brownian motion, $\lambda(D)$ determines the {long-time} behavior of the heat kernel and the survival probability:
\begin{equation}
 \lim_{t\rightarrow\infty}\frac{\log \mathbb{P}_x(\tau_D > t)}{t}=-\lambda(D),\ \text{for $\mu$-{a.e.} $x\in D$,} \label{pt:long}
\end{equation}
where $\tau_D$ is the first time of exiting $D$, and $\mathbb{P}_x(\tau_D > t)$ is the survival probability starting at $x$ and has not {left} $D$ up to time $t$. {We refer to \cite[Theorem 10.24]{Gri09} and \cite[Theorem 4.25]{Dav90} for similar results on weighted manifolds and on spaces with finite total mass, respectively.} Mariano and Wang \cite[Proposition 5.2]{MW25} prove that \eqref{pt:long} holds for {\emph{every} open subset} $D$ of unbounded doubling spaces, under the irreducibility and the sub-Gaussian heat kernel upper bound \eqref{UEloc} of strongly local regular Dirichlet forms. Moreover, they obtain in \cite[Theorem 3.1]{MW25} the following {asymptotic} estimate: 
\begin{equation}
\esup_{x \in D} \mathbb{P}_x(\tau_D > t) \leq 
C\left(1+\frac{2\lambda(D)}{d}t\right)^{d}\exp\left(-t\lambda (D)\right),\label{eq:th5-2}
\end{equation}
for some positive constants $C$ and $d$ independent of $t$ and $D$. 

In our framework, by following the proof of \cite{MW25} and using the stable-like heat kernel upper bound, a weaker upper bound for survival probability can be shown in Theorem \ref{thm:5}, that is
\begin{equation}
\esup_{x \in D} \mathbb{P}_x(\tau_D > t) 
 \leq C\exp\left(-\frac{1}{2}\kappa t\lambda (D)\right),\label{eq:th5-1}
\end{equation}
for some constants $\kappa\in (0,1)$ and $C>0$. 

The metric spaces we {work} with may be bounded or unbounded. We also emphasize the following novelties of this paper:
\begin{itemize}
\item Our results cover both diffusion and jump processes. Analytically, we work with regular Dirichlet forms without killing part.
\item Our main result is stated and proved for a general volume function $V(x,r)$ satisfying the volume doubling condition as well as for a general scale function $W(x,r)$ (that replaces $r^{\beta}$) that may depend on the position $x$, which covers many examples of metric measure spaces.
\end{itemize}

The structure of this paper is as follows. In Section \ref{Sect-1}, we introduce some basic conditions and our main results. In Section \ref{Sect-3}, we obtain the large time heat kernel upper bound and prove Theorem \ref{thm:5}. In Section \ref{Sect-4}, we prove Theorem \ref{t.main}. In Section \ref{s.DisEm}, we show that \eqref{EM} is preserved under quasi-symmetric changes of metric on uniformly perfect metric spaces and give {three} examples to illustrate {our results}.

\begin{notation}
	Throughout this paper, we use the following notation.
	\begin{enumerate}[label=\textup{({\arabic*})},align=right,leftmargin=*,topsep=5pt,parsep=0pt,itemsep=2pt]
		\item  The symbols $\subset$ and $\supset$ for set inclusion allow the case of equality.
 		\item Let $X$ be a non-empty set. We define $\one_{A}:X\to\{0,1\}$ for $A\subset X$ by
	 \[\one_{A}(x):= \begin{cases}
	 	1 & \mbox{if $x \in A$,}\\
	 	0 & \mbox{if $x \notin A$.}
	 \end{cases} \]
	 \item Let $U$ and $V$ be two open subsets in a topological space, if $U$ is precompact and the closure of $U$ is contained in $V$, then we write $U\Subset V$.
		\item Let $X$ be a topological space. We set
		$C(X):=\{\textrm{$f:X\to\mathbb{R}$ and $f$ is continuous}\}$ and
		$C_c(X):=\{f\in C(X): \textrm{$X\setminus f^{-1}(0)$ has compact closure in $X$}\}$.
		\item In a metric space $(X,d)$, $B(x,r)=B_{d}(x,r)$ is the open ball centered at $x \in X$ of radius $r>0$. For $A\subset X$, the \emph{diameter} of $A$ is defined by $\diam(A):=\sup_{x,y\in A}d(x,y)$. 
	  \item Given a ball $B:=B(x,r)$ and $K>0$, by $KB$ we denote the ball $B(x,Kr)$.
	\end{enumerate}
\end{notation}

\section{The framework and main results}\label{Sect-1}
{Let $(M,d,\mu)$ be a metric measure space, that is, $(M,d)$ is a locally compact, separable metric space and $\mu $ is a Radon measure with full support on $M$. Assume furthermore that $M$ contains at least two points.} We always assume that every metric ball $B(x,r)$ is \emph{precompact}. We fix throughout this paper a number $\overline{R}:=\diam(M)\in(0,\infty]$.

We consider a regular Dirichlet form $(\mathcal{E},\mathcal{F})$ on $L^{2}(M,\mu)$ without killing part, which can be written by the Beurling--Deny decomposition \cite[Theorems 3.2.1 and 4.5.2]{FOT11} as follows:
\begin{equation}
\mathcal{E}(u,v)=\mathcal{E}^{(L)}(u,v)+\mathcal{E}^{(J)}(u,v),\ \text{for all } u,v\in \mathcal{F},
\end{equation}
where $\mathcal{E}^{(L)}$ is the \emph{local part} and $\mathcal{E}^{(J)}$ is the \emph{jump part} determined by a unique Radon measure $j$ on $(M\times M)\setminus\diag$:
\begin{equation}
\mathcal{E}^{(J)}(u,v)=\underset{(M\times M)\setminus\diag}{\int}(u(x)-u(y))(v(x)-v(y)) j(\dif x,\dif y), \ \text{for all } u,v\in \mathcal{F}\cap C_{c}(M).\label{EJ}
\end{equation}
{Here} $\diag:=\{(x,x):x\in M\}$.

A measurable function $J(x,y)$ that is \emph{pointwise defined} on $(M\times M)\setminus\diag$ is called a \emph{jump kernel} if 
\begin{equation*}
 j(\dif x,\dif y)=J(x,y)\mu (\dif y)\mu (\dif x)\text{ \ on }(M\times M)\setminus\diag.
\end{equation*}

By \cite[Section 1.3]{FOT11}, there is a unique strongly continuous Markovian semigroup $\{P_{t}\}_{t\geq 0}$ on $L^{2}(M,\mu)$ that corresponds to $(\mathcal{E},\mathcal{F})$. Since the Dirichlet form is regular, Fukushima's theorem \cite[Theorem 7.2.1]{FOT11} gives a $\mu$-symmetric Hunt process $X=(\{X_{t}\}_{t\geq0}, \{\mathbb{P}_{x}\}_{x\in M})$ on $M$ associated with $(\mathcal{E},\mathcal{F})$, in the sense that for any $t>0$ and any bounded and $L^{2}$-integrable function $u$ on $M$, \begin{equation}
x\mapsto \mathbb{E}^{x}[u(X_{t})] \text{ is a \textit{$\mathcal{E}$-quasi-continuous $\mu$-{version}} of } P_{t}u.
\end{equation}

Recall $\overline{R}$ is defined as the diameter of $M$ under metric $d$.
\begin{definition}\label{d.vd+rvd}
Let $(M,d)$ be a metric space and let $\mu$ be a Borel measure on $M$.
\begin{enumerate}[label=\textup{({\arabic*})},align=right,leftmargin=*,topsep=5pt,parsep=0pt,itemsep=2pt]
	\item We say that $\mu$ satisfies the \emph{volume doubling property} \eqref{VD}, if there exists $C_{\mathrm{VD}} \in (1,\infty)$ such that for all $x \in M$ and $r>0$,
	\begin{equation} \label{VD} \tag{$\operatorname{VD}$}
	0 < V(x,2r) \le C_{\mathrm{VD}} V(x,r) <\infty,
	\end{equation}
where $V(x,r):=\mu (B(x,r))$ for the open metric ball $B(x,r)$ defined by 
\begin{equation*}
B(x,r):=\{y\in M:d(y,x)<r\}.
\end{equation*}
	\item We say that $\mu$ satisfies the \emph{reverse volume doubling property} \eqref{RVD}, if there exist $C_{\mathrm{RVD}}  \in (0,1], \alpha_{1} \in (0,\infty)$ such that for all $x\in M$ and $0<r \le R  <\overline{R}$,
	\begin{equation} \label{RVD} \tag{$\operatorname{RVD}$}
	V(x,R) \ge C_{\mathrm{RVD}} \left(\frac{R}{r}\right)^{\alpha_{1}} V(x,r).
	\end{equation}
\end{enumerate}
\end{definition}

We remark that if $\mu$ satisfies the {volume doubling property} \eqref{VD}, then for all $x,y\in M$ and $0<r\leq R<\infty $,
	\begin{equation} \label{e.VDip}
	\frac{V(x,R)}{V(y,r)} \leq C_{\mathrm{VD}} \left( \frac{d(x,y)+R}{r}\right)^{\alpha},
	\end{equation}
where $\alpha=\log_2 C_{\mathrm{VD}}$, see for example \cite[Proposition 5.1]{GH14}.

A function $W:M\times \lbrack 0,\infty )\rightarrow \lbrack 0,\infty)$ is a \emph{scale function}, if
\begin{itemize}
\item for each point $x\in M$, $W(x,\cdot )$ is strictly increasing with $W(x,0)=0$ and 
\begin{equation}
{W(x,\infty):=\lim_{r\uparrow\infty}  W(x,r)=\infty}; \label{e.w}
\end{equation}
\item there exist positive constants $C\geq1$
and $\beta _{2}\geq \beta _{1}$ such that for all $0<r\leq R<\infty $ and
all $x,y\in M$ with $d(x,y)\leq R$,
\begin{equation}
C^{-1}\left( \frac{R}{r}\right) ^{\beta _{1}}\leq \frac{W(x,R)}{W(y,r)}\leq
C\left( \frac{R}{r}\right) ^{\beta _{2}}.  \label{eq:vol_0}
\end{equation}
\end{itemize}
Let $W^{-1}(x,\cdot )$ be the inverse function of $W(x,\cdot )$. It is easy to see that 
for all $0<r\leq R<\infty $ and all $x\in M$, 
\begin{equation}
C^{-1/\beta _{2}}\left( \frac{R}{r}\right) ^{1/\beta _{2}}\leq \frac{%
W^{-1}(x,R)}{W^{-1}(x,r)}\leq C^{1/\beta _{1}}\left( \frac{R}{r}\right)
^{1/\beta _{1}}.  \label{w-1}
\end{equation}

We say that condition \eqref{C} holds if $(\mathcal{E},\mathcal{F})$ is \emph{conservative}, that is \begin{equation}\label{C}\tag{$\operatorname{C}$}
P_{t}\one_{M} =\one_{M},\ \text{for each}\ t>0,
\end{equation}
where $P_{t}\one_{M}$ is defined by {extending} $\{P_{t}\}$ from $L^{2}(M,\mu)$ to $L^{\infty}(M,\mu)$ using the Markovian property.

{Next, we introduce the definition of \emph{$\mathcal{E}$-nest} (see \cite[Section 2.1]{FOT11} for details). For any open subset $\Omega\subset M$, define the \emph{$1$-capacity} of $\Omega$ by
\begin{equation*}
\mathrm{Cap}_1(\Omega) := \inf \left\{\mathcal{E}_{1}(u,u): u\in \mathcal{F} \text{ and }u\geq 1 \ \ \mu\text{-a.e. on } \Omega\right\},
\end{equation*}
where $\mathcal{E}_{1}(\cdot ,\cdot ):=\mathcal{E}(\cdot ,\cdot )+\langle\cdot
,\cdot \rangle_{L^{2}(M,\mu)}$. An increasing sequence of closed subsets $\{F_k\}_{k=1}^\infty$ of $M$ is called an \emph{$\mathcal{E}$-nest} on $M$ if $\lim_{k\rightarrow\infty} \mathrm{Cap}_1(M\setminus F_k)=0$.}

The following definition of \emph{pointwise heat kernel} is from \cite[Definition 6.1]{GHH24b}. 
\begin{definition}\label{phk}
Let $(\mathcal{E}, \mathcal{F})$ be a Dirichlet form on $L^{2}(M,\mu)$ and let $\{P_{t}\}_{t>0}$ be the associated strongly continuous Markovian semigroup on $L^{2}(M,\mu)$. We say a function $p_t(x, y)$ of three variables $(t, x, y) \in (0,\infty
)\times M\times M$ is a \emph{pointwise heat kernel} of $(\mathcal{E}, \mathcal{F})$ if there exists a regular $\mathcal{E}$-nest $\{F_k \}_{k=1}^\infty$ (independent of $t$) such that the following statements are true for \emph{all} $t,s > 0$ and \emph{all} $x, y \in M$.
\begin{enumerate}[label=\textup{(HK{\arabic*})},align=right,leftmargin=*,topsep=5pt,parsep=0pt,itemsep=2pt]
\item\label{it.HK1} If one of {the} points $x,y$ lies outside $\bigcup_{k=1}^{\infty}F_k$, then $p_t(x,y)=0$.

\item\label{it.HK2} (\emph{Quasi-continuity}): $\displaystyle p_t(x,\cdot)\big|_{F_k}$ is continuous for each $k$.

\item\label{it.HK3} (\emph{Measurability}): $p_t(\cdot,\cdot)$ is jointly measurable on $M\times M$.

\item\label{it.HK4} (\emph{Markov property}): $p_t(x, y)\geq 0$ and $\displaystyle\int_Mp_t(x, y)\mu(\dif y)\leq 1$.

\item\label{it.HK5} (\emph{Symmetry}): $p_t(x, y) = p_t(y, x)$.

\item\label{it.HK6} (\emph{Semigroup property}): $\displaystyle p_{s+t}(x,y)=\int_Mp_s(x,z)p_t(z,y)\mu(\dif z)$.

\item\label{it.HK7} For any $f\in  L^2(M,\mu)$ and any $t>0$, we have
\begin{equation}
\text{the function } \int_Mp_t(\cdot,y)f(y)\mu(\dif y)\;\bigg|_{F_{k}}\ \text{is continuous for each $k$}
\end{equation}
and
\begin{equation}
P_tf(x)=\int_{M}p_t(x,y)f(y)\mu(\dif y),\  \text{for $\mu$-{a.e.}} \ x\in M.
\end{equation}
\end{enumerate}
\end{definition}
\emph{Throughout this paper, a heat kernel always refers to a pointwise heat kernel}.
\begin{definition}
We say that condition \eqref{UE} holds if there exists a heat kernel $p_{t}(x,y)$ and $C_{\mathrm{UE}} > 0$ such that for all $x, y \in M$ and all $0 < t < W(x,\overline{R})\wedge W(y,\overline{R})$,
\begin{equation}\label{UE}\tag{$\operatorname{UE}$}
p_t(x,y) \leq C_{\mathrm{UE}}\left( \frac{1}{V(x, W^{-1}(x,t))} \wedge \frac{t}{V(x,
d(x,y)){ W(x,d(x,y))}} \right).
\end{equation}
\end{definition}

We define the \emph{on-diagonal upper estimate} of the heat kernel.
\begin{definition}
We say that condition \eqref{DUE} holds
if there exist a heat kernel $p_{t}(x,y)$ and $C_{\mathrm{DUE}}>0$ such that for all 
$x\in M$ and all $0<t<W(x,\overline{R})$,
\begin{equation}\label{DUE} \tag{$\operatorname{DUE}$}
p_t(x,x) \leq \frac{C_{\mathrm{DUE}}}{V(x, W^{-1}(x,t))}.
\end{equation}
\end{definition}
An equivalent formulation of \eqref{DUE} is the following: there exist a heat kernel $p_{t}(x,y)$ and $C_{\mathrm{DUE}}>0$ such that for all $x,y\in M$ and all $0<t<W(x,\overline{R})\wedge W(y,\overline{R})$,
\begin{equation}\label{e.DUE+}
p_t(x,y) \leq \frac{C_{\mathrm{DUE}}}{\sqrt{V(x, W^{-1}(x,t))}\sqrt{V(y, W^{-1}(y,t))}}.
\end{equation}

For any non-empty open subset $\Omega$ of $M$, denote by $\mathcal{F}(\Omega)$ the closure of $\mathcal{F}\cap C_{c}(\Omega)$
 in the norm $\sqrt{\mathcal{E}_{1}}$. The form $(\mathcal{E},\mathcal{F}(\Omega))$ is a regular
Dirichlet form in $L^{2}(\Omega,\mu )$ if $(\mathcal{E},\mathcal{F})$ is regular (see \cite[Theorem 4.4.3]{FOT11}). The Hunt process $X^{\Omega}$ corresponding to $(\mathcal{E},\mathcal{F}(\Omega))$ is the \emph{part {process}} of $X$ on $\Omega$ \cite[Theorem 3.3.8]{CF12}, that is 
\begin{equation}\label{e.Huntpart}
X^{\Omega}_{t}=\begin{dcases}
X_{t},\quad &\text{if } t<\tau_{\Omega}\\
\Delta,\quad &\text{if }t\geq\tau_{\Omega},
\end{dcases}
\end{equation}
where $\Delta$ is the \emph{cemetery point} in the theory of Markov process \cite[Appendix A]{CF12} and $\tau_{\Omega}$ is the first exit time defined by 
\begin{equation}
\tau_{\Omega}:=\inf\{t>0:X_{t}\in M\setminus \Omega\}. \label{exit}
\end{equation}
Denote by $\{P_t^{\Omega}\}_{t\geq0}$ the semigroups of $(\mathcal{E},\mathcal{F}(\Omega))$, and by $p_t^{\Omega}(x, y)$ the
corresponding heat kernel. Denote the \emph{{bottom of the spectrum of $\Omega$}} by 
\begin{equation}
\lambda(\Omega):=\inf_{u\in \mathcal{F}(\Omega)\setminus \{0\}}\frac{\mathcal{E}%
(u)}{\lVert u\rVert_{L^{2}(\Omega,\mu)}^{2}}.  \label{lam1}
\end{equation}

{In the following several conditions, the letter $\sigma$ is used to denote the small numbers in $(0, 1)$ in several different conditions, and we will assume their values are the same, otherwise we take the minimum of them.}

\begin{definition}
We say that condition \eqref{lambda} holds if there exist constants $C\geq1$ and $\sigma
\in (0,1)$ such that for all $x\in M$ and $r\in (0,\sigma \overline{R })$,
\begin{equation}
C^{-1} W\left(x,r\right)^{-1} \leq \lambda\left(B\left(x,r\right)\right)\leq C{W\left(x,r\right)^{-1}}.\label{lambda}\tag{$\operatorname{\lambda}$}
\end{equation}
{We say that conditions (\hypertarget{L>}{}\hyperlink{L>}{$\operatorname{\lambda_{\geq}}$}) and (\hypertarget{L<}{}\hyperlink{L<}{$\operatorname{\lambda_{\leq}}$}) hold if the lower bound and upper bound in \eqref{lambda} hold, respectively.}
\end{definition}

\begin{definition}[Faber--Krahn inequality]
We say that condition \eqref{FK} holds if there exist three positive constants $C$, $\nu \in (0,1)$ and $\sigma
\in (0,1)$ such that for any ball $B(x,r)$ with $r\in(0,\sigma \overline{R})$ and any non-empty open subset $\Omega\subset B(x,r)$, 
\begin{equation}\tag{$\operatorname{FK}$}
\lambda(\Omega)\geq \frac{C}{W(x,r)}\left( \frac{V(x,r)}{\mu (\Omega)}%
\right) ^{\nu }.  \label{FK}
\end{equation}
\end{definition}

Clearly, \eqref{FK} implies (\hyperlink{L>}{$\operatorname{\lambda_{\geq}}$}) by taking $\Omega=B(x,r)$ in \eqref{FK}, that is
\begin{equation}
\eqref{FK} \Longrightarrow (\hyperlink{L>}{\operatorname{\lambda_{\geq}}}). \label{FK-1}
\end{equation}

We say that a Borel set $\mathcal{N}\subset M$ is \emph{properly exceptional} (see \cite[p. 153]{FOT11}) if 
 \begin{equation}
 \mu(\mathcal{N})=0 \text{ and }\mathbb{P}_{x}(X_{t}\in\mathcal{N} \text{ and }X_{t^{-}}\in\mathcal{N}\text{ for some $t>0$})=0\  \text{for all}\ x\in M\setminus \mathcal{N},
 \end{equation}
{where $X_{t^{-}}:=\lim_{s\uparrow t}X_s$}.

\begin{definition}[Mean exit time estimate]
 We say that condition $\eqref{E}$ holds if there exist $C_{\mathrm{E}}\geq1$, $\sigma \in (0,1)$ and a properly exceptional set 
 $\mathcal{N}\subset M$ such that for all $x\in M\setminus \mathcal{N}$ and $r\in (0,\sigma \overline{R })$,
\begin{equation}
C_{\mathrm{E}}^{-1}W(x,r)\leq \mathbb{E}_{x}\left[\tau_{B(x,r)}\right]\leq C_{\mathrm{E}}W(x,r).\tag{$\operatorname{E}$}  \label{E}
\end{equation}
We say that {conditions \textup{(\hypertarget{E>}{}\hyperlink{E>}{$\operatorname{\mathrm{E}_{\geq}}$})} and \textup{(\hypertarget{E<}{}\hyperlink{E<}{$\operatorname{\mathrm{E}_{\leq}}$})} hold if the lower bound and the upper bound in \eqref{E} hold\footnote{{We use the convention that mean exit times are allowed to take the value \(\infty\). Inequalities involving mean exit times are understood in this extended sense.}}, respectively.}
\end{definition}

\begin{definition}
\begin{enumerate}[label=\textup{({\arabic*})},align=right,leftmargin=*,topsep=5pt,parsep=0pt,itemsep=2pt]
\item {We say that condition \eqref{Eb} holds if there exist
$C>0$, $\sigma \in (0,1)$ and a properly exceptional set $\mathcal{N}\subset M$ such that for all $x\in M\setminus \mathcal{N}$ and $r\in (0,\sigma \overline{R })$,
\begin{equation}\label{Eb}\tag{$\operatorname{\bar{\mathrm{E}}}$}
C\sup_{y\in B\left(x,r\right)\setminus\mathcal{N}}\mathbb{E}_{y}\left[\tau_{B\left(x,r\right)}\right]
\leq\mathbb{E}_{x}\left[\tau_{B\left(x,r\right)}\right].
\end{equation} }
\item We say that condition \eqref{EM} holds if there exist $C>0$, $A>1$, $\sigma \in (0,1)$ and a properly exceptional set $\mathcal{N}\subset M$ such that for all $x\in M\setminus \mathcal{N}$ and $r\in (0,\sigma \overline{R })$,
\begin{equation}\label{EM}\tag{$\operatorname{\mathrm{E_M}}$}
C\sup_{y\in B\left(x,r\right)\setminus\mathcal{N}}\mathbb{E}_{y}\left[\tau_{B\left(x,r\right)}\right]
\leq\mathbb{E}_{x}\left[\tau_{B\left(x,Ar\right)}\right].
\end{equation}
\end{enumerate}
 \end{definition}

The  {Harnack-type inequality} \eqref{Eb} on mean exit time is introduced by Mariano and Wang \cite{MW25}. The condition \eqref{EM} we introduced in this paper is slightly different from \eqref{Eb}. The benefits of this modification will be discussed in Section \ref{s.DisEm}. {Conditions similar to \eqref{EM} were also studied by Telcs \cite{Tel06} in the setting of graphs, which are called (TD) and (TC) therein.}

\begin{definition}
We say that condition \eqref{J<} holds if the jump kernel exists and there exists a constant $C>0$ such that for any $x\in M$, \begin{equation}\label{J<}\tag{$\operatorname{\mathrm{J_{\leq}}}$}
J(x,y)\leq \frac{C}{V(x,d(x,y)){W(x,d(x,y))}},\, \text{for $\mu$-a.e. $y\in M$}.
\end{equation}
\end{definition}

The {first} main result in this paper is to characterize the heat kernel upper bound \eqref{UE} using \eqref{EM} .

\begin{theorem}\label{t.main}
Let $(\mathcal{E},\mathcal{F})$ be a regular Dirichlet form without killing part. Assume that condition \eqref{VD} holds, we have 
\begin{align}\label{eq:equ}
 \eqref{UE}+\eqref{C}+\eqref{RVD}&\Longleftrightarrow
 \eqref{FK}+\textup{(\hyperlink{L<}{$\operatorname{\lambda_{\leq}}$})}+\eqref{EM}+\eqref{J<}\\
 &\Longleftrightarrow {\eqref{FK}+\textup{(\hyperlink{L<}{$\operatorname{\lambda_{\leq}}$})}+\eqref{Eb}+\eqref{J<}}. \label{eq:eq+}
\end{align}
\end{theorem}

The {second} main result in this paper is to obtain an upper bound for survival probability through {the bottom of the spectrum}, given \eqref{VD} and \eqref{UE}.
\begin{theorem}
\label{thm:5} Let $(\mathcal{E}, \mathcal{F})$ be a regular Dirichlet form
on $L^{2}(M, \mu)$ without killing part. Assume that conditions \eqref{VD} and \eqref{UE} hold. Let $\kappa:=2\beta _{1}({\alpha +2\beta _{1}})^{-1}\in(0,1)$, where $\alpha$ comes from \eqref{e.VDip} and $\beta_1$ comes from \eqref{eq:vol_0}.
 Then there exist a properly exceptional set $\mathcal{N} \subset M$
and a constant $C>0$ such that for any $t>0$ and any open set $D
\subset M$, we have 
\begin{equation}
\sup_{x \in D \setminus \mathcal{N}} \mathbb{P}_x(\tau_D > t) \leq C\exp\left(-\frac{1}{2}\kappa t\lambda (D)\right).  \label{eq:th5}
\end{equation}
\end{theorem}
\begin{remark}
	{
We remark that a similar result was obtained in \cite[Theorem 2.2]{MW25+} with a constant of $\beta _{1}({\alpha +4\beta _{1}})^{-1}$ in the exponent by using different methods. Note that the constant of $\kappa/2=\beta _{1}({\alpha +2\beta _{1}})^{-1}$ in the exponent appearing in Theorem \ref{thm:5} is sharper than $\beta _{1}({\alpha +4\beta _{1}})^{-1}$. 
}
\end{remark}
\begin{remark}
{
The lower bound of the survival probability is a general fact, and is given in \cite[Proposition 4.6]{MW25} as follows. For any regular Dirichlet form $(\mathcal{E}, \mathcal{F})$ on $L^{2}(M, \mu)$, we have
\begin{equation*}
\esup_{x \in D } \mathbb{P}_x(\tau_D > t) \geq \exp\left(- \lambda (D)t\right).
\end{equation*}
}
\end{remark}
As a consequence of Theorem \ref{thm:5}, we see that there is an inverse proportionality between the mean exit time and {the bottom of the spectrum}.
\begin{corollary}[\text{\cite[Proposition 4.1]{MW25}}]Under the same assumptions as Theorem \ref{thm:5}, we have 
\begin{enumerate}[label=\textup{({\arabic*})},align=right,leftmargin=*,topsep=5pt,parsep=0pt,itemsep=2pt]
\item For any open set $D\subset M$,
\begin{equation}
\lambda\left(D\right)>0\text{ if and only if }\esup_{x\in D}\mathbb{E}_x\left[\tau_{D}\right]<\infty.
\end{equation}

\item For any $p>0,$ there exists a constant $C_{p}\geq\Gamma\left(p+1\right)$
such that for every open set $D\subset M$ satisfying $\lambda\left(D\right)>0$,
we have
\begin{equation}
\Gamma\left(p+1\right)\leq\lambda\left(D\right)^{p}\cdot\esup_{x\in D}\mathbb{E}_{x}\left[\tau_{D}^{p}\right]\leq C_{p},   \label{e.E1}
\end{equation}
{where $\Gamma(p+1):=\int_0^\infty t^{p}e^{-t}\mathrm{d}t$ is the Gamma function.}

\item For any open set $D\subset M$, if $\esup_{x\in D}\mathbb{E}_{x}\left[e^{a\tau_{D}}\right]<\infty$ {for some $a>0$}
then
\begin{equation}
\lambda\left(D\right)>a+\dfrac{a}{\esup_{x\in D}\mathbb{E}_{x}\left[e^{a\tau_{D}}\right]-1}.
\end{equation}
\end{enumerate}
\end{corollary}

\begin{remark}
{In the non-local setting of $\alpha$-stable processes, Giorgi and Smits proved \eqref{e.E1} with $p=1$ in \cite[Theorem 3.1]{GS10}, and Panzo showed the upper bound $C_p$ explicitly with $p=1$ in \cite[Theorem 2.1]{Pan22}. Also, Franzina proved \eqref{e.E1} with $p=1$ for the $(s, p)$-laplacian in \cite[(1.5) in Theorem 1.1]{Fra19}. In the setting of metric measure spaces, the lower bound of \eqref{e.E1} was given in \cite[Lemma 6.2]{GH14} without using \eqref{VD} and \eqref{UE}, that is, }
\begin{equation}\label{e.Gr}
{1\leq \lambda\left(D\right)\cdot\esup_{x\in D}\mathbb{E}_{x}\left[\tau_{D}\right].}
\end{equation}
\end{remark}

\section{Survival probability upper bounds}
\label{Sect-3}

To derive an upper bound estimate for the large time survival probability, we need to study the long-time behavior of the heat kernel. If $(M,d)$ is unbounded, then {by \eqref{e.w},} $W(x,\overline{R})=\infty$ for all $x\in M$, and the goal is already achieved by the expression of \eqref{UE} (or \eqref{DUE}). However, if $(M,d)$ is bounded, condition \eqref{UE} fails to provide long-time upper bounds for the heat kernel: it merely provides estimates for time less than $\inf_{x\in M}W(x,\overline{R})$. Nevertheless, in this case, the long-time upper bounds for the heat kernel can also be derived by employing the semigroup property \ref{it.HK6}. 
 
Write $T_{0}:=\inf_{x\in M}W(x,\overline{R})$, then we have the following.

\begin{lemma}\label{l.UE+}
Assume that conditions \eqref{VD} and \eqref{DUE} hold. If $T_{0}<\infty$, then there exists a constant $C>0$ {depending} only on $C_{\mathrm{DUE}}$ and $C_{\mathrm{VD}}$ such that 
\begin{equation}\label{e.UE+}
p_{t}(x,y)\leq \frac{C}{V(x,\overline{R})},\ \text{for all}\ x,y\in M\ \text{and}\ t\geq T_0.
\end{equation}
\end{lemma}
\begin{proof}
Recall that $T_{0}<\infty$ implies $\overline{R}<\infty$. By \eqref{eq:vol_0}, we know that there exists $C>1$ such that for all $x\in M$ 
\begin{equation}
  T_0\leq W(x,\overline{R})\leq CT_0. \label{T0}
\end{equation}
By \eqref{e.VDip} and \eqref{T0}, for all $x,y\in M$ and all $s\in (0,T_0]$, we have
\begin{align}
\frac{V(x, \overline{R})}{V(y, W^{-1}(y,s))}&  \leq C_{\mathrm{VD}}\left( \frac{d(x,y)+\overline{R}}{ W^{-1}(y,s)}\right)^{\alpha}
\leq C_{\mathrm{VD}}^2 \left(\frac{W^{-1}(y,W(y,\overline{R}))}{W^{-1}(y,s)}\right)^\alpha \\
&\leq  C_{\mathrm{VD}}^2 C^{\frac{\alpha}{\beta_1}}\left(\frac{W(y,\overline{R})}{s}\right)^{\frac{\alpha}{\beta_1}}
\leq C_{\mathrm{VD}}^2 C^{\frac{2\alpha}{\beta_1}} \left(\frac{T_0}{s}\right)^{\frac{\alpha}{\beta_1}}.  \label{p-1}
\end{align}

For any measurable subset $E\subset M$, we have, by Fubini's theorem and the symmetry of heat kernel, that for any 
$x\in M$ and all $t\geq T_0$,
\begin{align}
&\phantom{\ \ \ {\leq}}\int_{E}p_{t}(x,y)\mu(\dif y)\overset{\ref{it.HK6}}{=}\int_{M}p_{t-T_0/2}(x,z)\left(\int_{E}p_{T_0/2}(y,z)\mu(\dif y)\right)\mu(\dif z)\\
&\overset{\eqref{e.DUE+}}{\leq} \int_{M}p_{t-T_0/2}(x,z)\left(\int_{E}\frac{C_{\mathrm{DUE}}}{V(y,W^{-1}(y,T_0/2))^{1/2}V(z,W^{-1}(z,T_0/2))^{1/2}}\mu(\dif y)\right)\mu(\dif z)\\
&\overset{\eqref{p-1}}{\leq}\int_{M}p_{t-T_0/2}(x,z)\left(\int_{E}\frac{C_{\mathrm{DUE}}C_{\mathrm{VD}}^2 C^{\frac{2\alpha}{\beta_1}}2^{\frac{\alpha}{\beta_1}}}{V(x, \overline{R})}\mu(\dif y)\right)\mu(\dif z)\\
&\overset{\ref{it.HK4}}{\leq} \int_{E}\frac{C_{\mathrm{DUE}}C_{\mathrm{VD}}^2 C^{\frac{2\alpha}{\beta_1}}2^{\frac{\alpha}{\beta_1}}}{V(x, \overline{R})}\mu(\dif y).\label{e.UE+2}
\end{align}
Since $E$ is arbitrary, \eqref{e.UE+2} implies \[\esup_{y\in M} p_{t}(x,y)\leq \frac{C_{\mathrm{DUE}}C_{\mathrm{VD}}^2 C^{\frac{2\alpha}{\beta_1}}2^{\frac{\alpha}{\beta_1}}}{V(x, \overline{R})}.\]
By \ref{it.HK1} and the quasi-continuity of heat kernel in \ref{it.HK2}, we obtain \eqref{e.UE+}.
\end{proof}

The following lemma indicates that condition \eqref{DUE} {ensures that the heat kernel $p^{D}_t(x,y)$ of Dirichlet form $(\mathcal{E},\mathcal{F}(D))$ exists} for any open set $D\subset M$.
\begin{proposition}
[\text{\cite[Corollary 6.9]{GHH24b}}]
\label{p.L3} If conditions \eqref{VD} and \eqref{DUE} hold, then {a} heat kernel $p_t(x,y)$ exists, and for any non-empty open set $D\subset M$, {a} heat kernel $p^{D}_t(x,y)$ of Dirichlet form $(\mathcal{E},\mathcal{F}(D))$ exists. Moreover, for all $x,y \in M$ and $t>0$,
\begin{equation}\label{e.L3}
p^{D}_t(x,y) \leq p_t(x,y).
\end{equation}
\end{proposition}

\begin{corollary}\label{c.prop1}
If conditions \eqref{VD} and \eqref{DUE} hold, then for any open set $D\subset M$, 
\begin{equation}
  p^D_t(\cdot,x)\in L^2(D,\mu),\ \text{for all }(t,x)\in (0,\infty)\times D.\label{e.p341}
\end{equation}
\end{corollary}
\begin{proof}
By {Proposition} \ref{p.L3} we know the heat kernel $p^{D}_{t}(x,y)$ exists. Then by \eqref{DUE} and \eqref{e.UE+} in Lemma \ref{l.UE+}, we have
\begin{equation*}
\Vert p^{D}_t(\cdot ,x)\Vert _{L^{2}(D,\mu)}^{2}\overset{\eqref{e.L3}}{\leq} \Vert p_t(\cdot ,x)\Vert
_{L^{2}(M,\mu)}^{2}\overset{\ref{it.HK6}}{=}p_{2t}(x,x)\leq 	\begin{dcases}
		\frac{C_{\mathrm{DUE}}}{V(x, W^{-1}(x,2t))}\  & \text{if}\ 2 t< T_0, \\ 
		\frac{C}{V(x,\overline{R})}\  & \text{if}\ 2 t\geq T_0,
	\end{dcases}
\end{equation*}
which gives \eqref{e.p341}.
\end{proof}

\begin{theorem}
\label{thm4} 
Let $\kappa:=2\beta _{1}({\alpha +2\beta _{1}})^{-1}\in(0,1)$, {where $\alpha$ and $\beta_{1}$ are constants appearing in \eqref{e.VDip} and \eqref{eq:vol_0}, respectively}. Let $\epsilon \in (0,\kappa)$, and $\delta \in (0,1)$. If conditions \eqref{VD}, \eqref{UE} hold, then there exists a constant $C=C_{\epsilon ,\delta}>0$ depending only on $\epsilon$, $\delta$ and the constants {appearing} in \eqref{VD} and \eqref{UE} such that, for any open set $D\subset M$,
\begin{equation}
\int_{D}{p^D_{t}}(x,y)\mu(\dif y)\leq C \exp (-\epsilon (1-\delta )t\lambda (D)),\ \text{{for all $(t,x)\in(0,\infty)\times D$}}.\label{eq:p4}
\end{equation}
\end{theorem}

\begin{proof}
Fix $x\in D$.
For any $\epsilon \in (0,1)$, by Hölder's inequality we have
\begin{align}
\int_{D}p^D_t(x,y)\mu (\dif y) &=\int_{D}p^D_t(x,y)^{\epsilon }\cdot
p^D_t(x,y)^{1-\epsilon }\mu (\dif y) \\
&\leq \left( \int_{D}p^D_t(x,y)^{2}\mu (\dif y)\right) ^{\frac{\epsilon }{2}}\cdot \left( \int_{D}p^D_t(x,y)^{\frac{2(1-\epsilon )}{2-\epsilon }}\mu (\dif y)\right) ^{\frac{2-\epsilon }{2}} \\
&=:I_{1}(t)^{\epsilon /2}\cdot I_{2}(t)^{\frac{2-\epsilon }{2}}.\label{e.p4}
\end{align}

Since $p^D_t\in L^2(D,\mu)$ by Corollary \ref{c.prop1}, we know by \cite[Lemma 3.3]{MW25} that 
\begin{align}
I_{1}(t)=\int_{D}p^D_t(x,y)^{2}\mu (\dif y)&\leq\exp(-2(1-\delta )t\lambda
(D))\int_{D}p^{D}_{\delta t}(x,y)^{2}\mu (\dif y) \\
&\overset{\ref{it.HK4}}{\leq}\exp(-2(1-\delta )t\lambda (D))\esup_{y\in D}p^{D}_{\delta t}(x,y) \\
&\leq \exp(-2(1-\delta )t\lambda (D))\esup_{y\in D}p_{\delta t}(x,y),\label{e.p4i-1}
\end{align}
where in the last inequality we have used {Proposition} \ref{p.L3}. 

\begin{enumerate}[label=\textit{Step {\arabic*}.},align=right,leftmargin=*,topsep=5pt, parsep=0pt, itemsep=2pt]
\item\label{it.t4S1} Recall that $T_{0}=\inf_{x\in M}W(x,\overline{R})$. If $0<t<T_0$, then we can estimate $I_{1}(t)$ by using condition \eqref{UE}. By \eqref{e.p4i-1}, we have
\begin{align}
I_{1}(t)&\leq \exp(-2(1-\delta )t\lambda (D))\esup_{y\in D}p_{\delta t}(x,y) \\
&\overset{\eqref{UE}}{\leq} \exp(-2(1-\delta )t\lambda (D))\cdot \frac{C}{V(x,W^{-1}(x,\delta t))}.\label{e.p4i1}
\end{align}
{Next, we} estimate $I_{2}(t)$. Let $q=\frac{2}{2-\epsilon }$. {Write} \begin{align}
I_{21}(t)&:= \int_{D\cap B\left(x,W^{-1}(x,t)\right)}p^D_t(x,y)^{q(1-\epsilon)}\mu (\dif y)\\ \text { and } I_{22}(t)&:= \int_{D\setminus B\left(x,W^{-1}(x,t)\right)}p^D_t(x,y)^{q(1-\epsilon)}\mu (\dif y),
\end{align}
so that $I_{2}(t)=I_{21}(t)+I_{22}(t)$.

For $I_{21}(t)$, 
\begin{align}
I_{21}(t) &\overset{\eqref{UE}}{\leq} \frac{C_{\mathrm{UE}}^{q(1-\epsilon
)}}{V(x,W^{-1}(x, t))^{q(1-\epsilon )}}V\left(x,W^{-1}(x,t)\right) \\
&= C_{\mathrm{UE}}^{q(1-\epsilon)}V\left(x,W^{-1}(x,t)\right)^{\frac{\epsilon}{2-\epsilon}}. 
\label{e.p40}
\end{align}
For $I_{22}(t)$, note by \eqref{eq:vol_0} that for any $y\in M\setminus B(x, W^{-1}(x,t))$,
\begin{equation*}
  \frac{t}{W(x,y)}=\frac{W(x,W^{-1}(x,t))}{W(x,d(x,y))}\leq C{\left(\frac{W^{-1}(x,t)}{d(x,y)}\right)^{\beta _{1}}},
\end{equation*}
therefore, {by \eqref{UE},}
\begin{align}
I_{22}(t) &\leq C\int_{B(x,W^{-1}(x,t))^{c}}\frac{C_{\mathrm{UE}}^{q(1-\epsilon )}%
}{V(x,d(x,y))^{q(1-\epsilon )}}\left(\frac{W^{-1}(x,t)}{%
d(x,y)}\right) ^{\beta _{1}q(1-\epsilon )}\mu (\dif y) \\
&\leq C\sum_{i=0}^{\infty }\int_{B(x,2^{i+1}%
W^{-1}(x,t))\setminus B(x,2^{i}W^{-1}(x,t))}\frac{C_{\mathrm{UE}}^{q(1-\epsilon )}\left(2^{-i}\right) ^{\beta _{1}q(1-\epsilon )}}{%
V(x,d(x,y))^{q(1-\epsilon )}}\mu (\dif y) \\
&\leq CC_{\mathrm{UE}}^{q(1-\epsilon )}\sum_{i=0}^{\infty }\int_{B(x,2^{i+1}%
W^{-1}(x,t))}\frac{1}{V(x,2^{i}W^{-1}(x,t))^{q(1-\epsilon )}}\cdot 
\frac{\mu (\dif y)}{2^{i\beta _{1}q(1-\epsilon )}} \\
&\leq CC_{\mathrm{UE}}^{q(1-\epsilon )}C_{\mathrm{VD}}\sum_{i=0}^{\infty }\frac{1}{V(x,2^{i}%
W^{-1}(x,t))^{q(1-\epsilon )-1}}\cdot \frac{1}{2^{i\beta
_{1}q(1-\epsilon )}}.\label{e.p41}
\end{align}
By \eqref{VD},
\begin{equation}
V(x,2^{i}W^{-1}(x,t))\leq C_{\mathrm{VD}}V(x,W^{-1}(x,t))2^{i\alpha },
\end{equation}
which implies
\begin{align}
\frac{1}{V(x,2^{i}W^{-1}(x,t))^{q(1-\epsilon )-1}}&=V(x,2^{i}W^{-1}(x,t)
)^{\frac{\epsilon }{2-\epsilon }}\\
&\leq (2^{i\alpha }C_{\mathrm{VD}})^{\frac{\epsilon }{%
2-\epsilon }}V(x,W^{-1}(x,t))^{%
\frac{\epsilon }{2-\epsilon }}.
\end{align}
Combining this and \eqref{e.p41}, we have
\begin{align*}
I_{22}(t)&\leq C\sum_{i=0}^{\infty }{2^{-\frac{%
i(1-\epsilon )}{2-\epsilon }\left( 2\beta _{1}-\frac{\alpha
\epsilon }{1-\epsilon }\right) }} V(x,W^{-1}(x,t))^{\frac{%
\epsilon }{2-\epsilon }}.
\end{align*}
Since $0<\epsilon <\frac{2\beta _{1}}{\alpha +2\beta _{1}}$, we have%
\begin{equation*}
I_{22}(t)\leq \frac{C}{1-2^{{-}\frac{1-\epsilon }{%
2-\epsilon }\left( 2\beta _{1}-\frac{\alpha \epsilon }{1-\epsilon }%
\right) }}V(x,W^{-1}(x,t))^{\frac{\epsilon }{2-\epsilon }},
\end{equation*}
which, combined with \eqref{e.p4} and \eqref{e.p40}, gives \eqref{eq:p4} for $t\in(0,T_{0})$. 

\item\label{it.t4S2} If $T_{0}=\infty$, then \eqref{eq:p4} is proved by \ref{it.t4S1} {If $T_{0}=\inf_{x\in M}W(x,\overline{R})<\infty$, then $\overline{R}=\diam(M)<\infty$ since $W(x,\cdot)$ is strictly increasing for each $x\in M$.} Therefore the metric space is bounded and $\mu(M)<\infty$ by the precompactness. {By} \eqref{UE}, \eqref{p-1} and Lemma \ref{l.UE+},
\begin{align}
\esup_{y\in D}p_{\delta t}(x,y)\leq 	\begin{dcases}
		\frac{C_{\mathrm{UE}}}{V(x, W^{-1}(x,\delta t))}\leq \frac{C\left(\frac{T_0}{\delta t}\right)^{\frac{\alpha}{\beta_1}}}{V(x,\overline{R})}\leq \frac{C\delta^{-\frac{\alpha}{\beta_1}}}{V(x,\overline{R})},\  & \text{if}\ \delta t< T_0, \\ 
		\frac{C}{V(x,\overline{R})},\  & \text{if}\ \delta t\geq T_0.
	\end{dcases}
\end{align}
Therefore, by \eqref{e.p4i-1}, we have
\begin{equation}
I_{1}(t)\leq C\exp(-2(1-\delta )t\lambda (D))\frac{\delta^{-\frac{\alpha}{\beta_1}}}{V(x,\overline{R})}. \label{e.i1-2}
\end{equation}
By Proposition \ref{p.L3} and Lemma \ref{l.UE+},
\begin{align}
 I_2(t)&= \int_{D}p^D_t(x,y)^{\frac{2(1-\epsilon )}{2-\epsilon }}\mu (\dif y)\\
 &\leq \int_{D}p_t(x,y)^{\frac{2(1-\epsilon )}{2-\epsilon }}\mu (\dif y)\leq \frac{C\mu(D)}{V(x,\overline{R})^{\frac{2(1-\epsilon )}{2-\epsilon }}}.\label{e.i2}
\end{align}
By \eqref{e.p4}, \eqref{e.i1-2} and \eqref{e.i2}, we obtain
\begin{align*}
 \int_{D}p_{t}^{{D}}(x,y)\mu(\dif y)&\leq C\exp(-\epsilon(1-\delta )t\lambda (D)) \delta^{-\frac{\epsilon\alpha}{2\beta_1}}
  \left(\frac{\mu(D)}{V(x,\overline{R})}\right)^{1-\epsilon/2} \\
  &\leq C\delta^{-\frac{\epsilon\alpha}{2\beta_1}}\exp(-\epsilon(1-\delta )t\lambda (D)),
\end{align*}
thus showing \eqref{eq:p4}.
\end{enumerate}
The proof is complete.
\end{proof}

\begin{lemma}\label{l.quasi=}
Let $(\mathcal{E}, \mathcal{F})$ be a regular Dirichlet form on $L^{2}(M,\mu)$ associated with Hunt process $X=(\{X_{t}\}_{t\geq0}, \{\mathbb{P}_{x}\}_{x\in M})$. Let $p_{t}(x,y)$ be a {pointwise heat kernel} of $(\mathcal{E}, \mathcal{F})$. Then for any $f\in L^{2}(M,\mu)$ and any $t>0$, there is a {properly} exceptional set $\mathcal{N}\subset M$ such that
\begin{equation}\label{e.quasi=}
\mathbb{E}^{x}[f(X_{t})]=\int_{M}p_t(x,y)f(y)\mu(\dif y), \text{ {for all} } x\in M\setminus \mathcal{N}.
\end{equation}
\end{lemma}
\begin{proof}
By \cite[Proposition 3.1.9 and Theorem 1.3.14]{CF12}, the function $x\mapsto \mathbb{E}^{x}[f(X_{t})]$ is $\mathcal{E}$-quasi-continuous and $P_{t}f(x)=\mathbb{E}^{x}[f(X_{t})]$ for $\mu$-a.e. $x\in M$. By \ref{it.HK7}, the function {$x\rightarrow \int_{M}p_t(x,y)f(y)\mu(\dif y)$} is also $\mathcal{E}$-quasi-continuous and is a $\mu$-version of $P_{t}f$. {Since} \begin{equation}
\mathbb{E}^{x}[f(X_{t})]=\int_{M}p_t(x,y)f(y)\mu(\dif y),\ \text{for $\mu$-a.e. $x\in M$},
\end{equation}
combining \cite[Lemma 2.1.5]{FOT11} and \cite[Theorems 3.1.3 and 3.1.5]{CF12}, we obtain \eqref{e.quasi=}.
\end{proof}

\begin{proof}[Proof of Theorem \ref{thm:5}]
{Applying} Lemma \ref{l.quasi=} to the regular Dirichlet form $(\mathcal{E},\mathcal{F}(D))$ with the part process {$X^D_t$} in \eqref{e.Huntpart} and with the indicator function $\one_{D}$, we see by \cite[Theorem 3.3.8]{CF12} that there is a properly exceptional set $\mathcal{N}_{t,D}\subset D$ such that 
\begin{equation}  \label{B}
\mathbb{P}_x(\tau_D > t) =\int_{D}p^{D}_{t}(x,y)\mu(\dif y), \text{ {for all} } x \in D\setminus \mathcal{N}_{t,D}.
\end{equation}

By Theorem \ref{thm4}, \eqref{B} implies that
\begin{equation}
\mathbb{P}_{x}(\tau_D > t) \leq C_{\epsilon ,\delta}
\exp\left(-\epsilon(1-\delta) t \lambda(D)\right),\text{ {for all} }  x \in D
\setminus \mathcal{N}_{t,D}.
\end{equation}

Let $\mathcal{V}$ be {a} countable topological basis for $M$ and consider 
\begin{equation*}
\mathcal{N} := \bigcup_{\substack{ V \in \mathcal{V}  \\ t \in \mathbb{Q}%
^+}} \mathcal{N}_{t,V},
\end{equation*}
then $\mathcal{N}$ is a properly exceptional set. We claim it is the desired
set. In fact, for any open set $D \subset M$, we can find $ \{D_n\} \subset \mathcal{V%
}$ such that $D_n \subset D_{n+1}$, $D = \bigcup_{n=1}^\infty D_n$. For any $x
\in D \setminus \mathcal{N}$, there exists $ m \geq 1$ such that $x \in D_k \setminus 
\mathcal{N}$ for all $k \geq m$. Hence, 
\begin{equation*}
x \in D_k \setminus {\mathcal{N}_{t,D_k}}, \quad \text{for all } t \in \mathbb{Q}^+,
\,  k \geq m.
\end{equation*}
By \eqref{B}, for any $t \in \mathbb{Q}^+$ and $k \geq m$, 
\begin{equation*}
\mathbb{P}_{x}(\tau_{D_k} > t)\leq C_{\epsilon ,\delta}
\exp\left(-\epsilon(1-\delta) t \lambda(D_k)\right).
\end{equation*}
Letting $k \to \infty$ and using the fact that $D_k \subset D$ implies $\lambda(D_k)\geq \lambda(D)$, we obtain 
\begin{equation}\label{e.pf21}
\mathbb{P}_x(\tau_D > t) \leq C_{\epsilon ,\delta}
\exp\left(-\epsilon(1-\delta) t \lambda(D)\right).
\end{equation}
{for all $t \in \mathbb{Q}^+$. A dense argument gives that \eqref{e.pf21} holds for all $t\in(0,\infty)$.}

Choosing $\delta=\frac{1}{3}$ and $\epsilon=\frac{3}{4}\kappa =\frac{3\beta_1}{{2(\alpha+2\beta_1)}}$ above, we know that
\begin{equation*}
\sup_{x\in D\setminus \mathcal{N}}\mathbb{P}_x(\tau_D > t)\leq C_{\frac{3}{4}\kappa,\frac{1}{3}}
\exp \left(-\frac{1}{2}\kappa t\lambda (D)\right), 
\end{equation*}
thus showing \eqref{eq:th5}. The proof is complete.
\end{proof}

\section{Characterization of heat kernel upper bounds}
\label{Sect-4}
{In this section, we will prove Theorem \ref{t.main}. First, we introduce generalized capacity condition and survival estimate. Prior to that, it is necessary to define certain terms.}

{For a fixed number $\kappa \geq 1$, a measurable function $%
\phi $ is a $\kappa $-\emph{cutoff function} for $U\Subset V$, denoted by $%
\phi \in \kappa $-$\mathrm{cutoff}(U,V)$, if $\phi \in \mathcal{F}$, $0\leq \phi
\leq \kappa ,\phi \geq 1\ \text{on}\ U,\text{ and }\phi =0\ \text{on}\ V^{c}$%
. When $\kappa =1$, we call the $1$-cutoff function the \emph{cutoff
function} and simply write $1$-$\mathrm{cutoff}(U,V)$ as $\mathrm{cutoff}(U,V)$. It is known
that $\mathrm{cutoff}(U,V)$ is non-empty for any two non-empty open subsets $%
U\Subset V$ of $M$, if $(\mathcal{E},\mathcal{F})$ is regular.}

{Define the vector space $\mathcal{F}^{\prime }$ by 
\begin{equation}
\mathcal{F}^{\prime }:=\{v+a:\ v\in \mathcal{F},\ a\in \mathbb{R}\}.
\label{ff1}
\end{equation}
We extend the domain $\mathcal{F}$ of $%
\mathcal{E}$ to the space $\mathcal{F}^{\prime }$ as follows: for any $%
u,v\in \mathcal{F}$ and any $a,b\in \mathbb{R}$, set $\mathcal{E}(u+a,v+b):=%
\mathcal{E}(u,v)$. We remark that this extension is well-defined for any
regular Dirichlet form without killing part.
\begin{definition}[\text{\cite[Definition 2.1]{GHH24c}}]
We say that the \emph{generalized capacity condition} \eqref{Gcap} is satisfied if there exist two numbers $\overline{\kappa }\geq1,C>0 $ such that, for any $u\in \mathcal{F}^{\prime}\cap L^{\infty }$ and
for any pair of concentric balls $B_{0}:=B(x_{0},R)$, $B:=B(x_{0},R+r)$ with 
$x_{0}\in M$ and $0<R<R+r<\overline{R}$, there exists some $\phi \in 
\overline{\kappa }$-$\mathrm{cutoff}(B_{0},B)$ such that 
\begin{equation}\tag{$\operatorname{Gcap}$}
\mathcal{E}(u^{2}\phi ,\phi )\leq \sup_{x\in B}\frac{C}{W(x,r)}%
\int_{B}u^{2}\dif\mu .  \label{Gcap}
\end{equation}
\end{definition}
\begin{definition}[\text{\cite[Definition 7.1]{GHH24c}}]We say that the \emph{survival estimate} \eqref{S} holds if there exist $\eta_1,\eta_2\in (0,1)$ such that for all $t>0$ and each ball $B=B(x_0,r)$ of radius $r<\overline{R}$ and $t\leq \eta_1W(x_0,r)$,
\begin{equation}\tag{$\operatorname{S}$}
  P^B_t\one_B\geq \eta_2\ \text{in $\frac{1}{4}B$}. \label{S}
\end{equation}
\end{definition}
We collect some useful known facts that will be used later in this paper.
\begin{proposition}\label{p.facts}
	Let $(\mathcal{E},\mathcal{F})$ be a regular Dirichlet form without killing part. \begin{enumerate}[label=\textup{({\arabic*})},align=right,leftmargin=*,topsep=5pt,parsep=0pt,itemsep=2pt]
	\item\label{it.fact2}  If condition \eqref{VD} holds, then \[\eqref{FK}+\eqref{Gcap}+\eqref{J<}\Longleftrightarrow\eqref{UE}+\eqref{C}+\eqref{RVD}\Longrightarrow \eqref{FK}+\eqref{S}.\]
	\item\label{it.fact3} $\eqref{E}\Longrightarrow \eqref{Eb}\Longrightarrow \eqref{EM}$.
	\item\label{it.fact4} $\eqref{VD}+\eqref{FK}+\eqref{Gcap} +\eqref{J<} \Longrightarrow \eqref{E}$.
	\item\label{it.fact5} $\eqref{FK}\Longrightarrow \textup{(\hyperlink{E<}{$\operatorname{\mathrm{E}_{\leq}}$})}$ and $\eqref{E}\Longrightarrow \eqref{Gcap}$.
\end{enumerate}
\end{proposition}
\begin{proof}
	\begin{enumerate}[label=\textup{({\arabic*})},align=right,leftmargin=*,topsep=5pt,parsep=0pt,itemsep=2pt]
	\item[\ref{it.fact2}] By the same proof of \cite[Lemma 6.9]{HL22}, we know that \[\eqref{VD}+\eqref{FK}+\eqref{J<} \Longrightarrow \eqref{RVD}.\] The claim is then a consequence of \cite[Theorem 2.15 for $q=\infty$, Lemma 10.3 and Proposition 10.4]{GHH24c}.
	\item[\ref{it.fact3}] This is proved in \cite[Lemma 7.12]{MW25}.
	\item[\ref{it.fact4}] See \cite[Lemma 11.2 and Corollary 12.5]{GHH24a}.
	\item[\ref{it.fact5}] See \cite[Lemma 12.2, Proposition 13.4 and Lemma 13.5]{GHH24a}.
\end{enumerate}
\end{proof}
\begin{lemma}
Let $(\mathcal{E},\mathcal{F})$ be a regular Dirichlet form without killing part. Then,
\begin{equation*}
 \eqref{VD}+\eqref{S}+\eqref{eq:th5}\Longrightarrow (\hyperlink{L<}{\operatorname{\lambda_{\leq}}}).
\end{equation*}
\end{lemma}
}
\begin{proof}
Let $B=B(x_0,r)$ with $r\in(0,\overline{R})$. {By condition \eqref{S}, we know that 
there exist $\eta_1,\eta_2\in (0,1)$ such that for all $t\leq \eta_1W(x_0,r)$,
\begin{equation*}
  P^B_t\one_B\geq \eta_2\ \text{in $\frac{1}{4}B$}.
\end{equation*}
}By \eqref{eq:th5}, we know that there exists $C>0$ such that
\begin{equation*}
  P^B_t\one_B(x)\leq C\exp\left(-\frac{1}{2}\kappa t\lambda(B)\right)\ \text{for $\mu$-a.e. $x\in \frac{1}{4}B$}.
\end{equation*}
Combining the above two inequalities, we know that for all $t\leq\eta_1W(x_0,r)$,
\begin{equation*}
 \eta_2 \leq C\exp\left(-\frac{1}{2}\kappa t\lambda(B)\right),
\end{equation*}
which implies
\begin{equation*}
  \eta_2\eta_1W(x_0,r)\leq C\int_0^{\eta_1W(x_0,r)}\exp\left(-\frac{1}{2}\kappa t\lambda(B)\right)\dif t.
\end{equation*}
Therefore,
\begin{equation*}
  \lambda(B)\leq \frac{2C}{\eta_1\eta_2\kappa }\frac{1}{W(x_0,r)},
\end{equation*}
so (\hyperlink{L<}{$\operatorname{\lambda_{\leq}}$}) holds.  The proof is complete.
\end{proof}

{Using \eqref{FK-1}, we obtain the following corollary as a consequence of Theorem \ref{thm:5} and Proposition \ref{p.facts}-\ref{it.fact2}.}

\begin{corollary}\label{L10}
Let $(\mathcal{E},\mathcal{F})$ be a regular Dirichlet form without killing part. Under condition \eqref{VD}, we have
\begin{equation*}
  \eqref{UE}+\eqref{C}+\eqref{RVD}\Longrightarrow \eqref{FK}+\eqref{lambda}.
\end{equation*}
\end{corollary}

The following assertion says that condition \textup{(\hyperlink{L<}{$\operatorname{\lambda_{\leq}}$})},
together with condition \eqref{EM}, will imply condition \textup{(\hyperlink{E>}{$\operatorname{\mathrm{E}_{\geq}}$})} {under \eqref{VD}}.
\begin{lemma}\label{L12}
Let $(\mathcal{E},\mathcal{F})$ be a regular Dirichlet form. Assume \eqref{VD}{. Then} we have 
\begin{equation*}
\textup{(\hyperlink{L<}{$\operatorname{\lambda_{\leq}}$})}+\eqref{EM}\Longrightarrow \textup{(\hyperlink{E>}{$\operatorname{\mathrm{E}_{\geq}}$})}.
\end{equation*}
\begin{proof}
By condition \eqref{EM}, there exist $C>0$, $A>1$, $\sigma \in (0,1)$ and a properly exceptional set $\mathcal{N}\subset M$ such that for all $x\in M\setminus \mathcal{N}$ and all $r\in(0,A^{-1}\sigma \overline{R})$,
\begin{equation}
C\sup_{y\in B\left(x,r\right)\setminus  \mathcal{N}}\mathbb{E}_{y}\left[\tau_{B\left(x,r\right)}\right]
\leq\mathbb{E}_{x}\left[\tau_{B\left(x,Ar\right)}\right]. \label{EM-1}
\end{equation}

For any $x\in M\setminus \mathcal{N}$ and $r\in(0,A^{-1}\sigma \overline{R})$, we know by {\eqref{e.Gr}} that 
  \begin{align*}
    1&\leq \lambda(B(x,r))\esup_{y\in B(x,r)}\mathbb{E}_y\left[\tau_{B(x,r)}\right] \\
    &\overset{\textup{(\hyperlink{L<}{$\operatorname{\lambda_{\leq}}$})}}{\leq} \frac{C}{W(x,r)}\sup_{y\in B(x,r)\setminus  \mathcal{N}}\mathbb{E}_y\left[\tau_{B(x,r)}\right] \overset{\eqref{EM-1}}{\leq}  \frac{C}{W(x,r)}\mathbb{E}_{x}\left[\tau_{B(x,Ar)}\right],
  \end{align*}
which implies \begin{equation}
\mathbb{E}_{x}\left[\tau_{B(x,Ar)}\right]\geq C^{-1}W(x,r)\geq C' A^{-\beta_2}W(x,Ar).
\end{equation} By adjusting the constants, we see that \textup{(\hyperlink{E>}{$\operatorname{\mathrm{E}_{\geq}}$})} holds. 
\end{proof}
\end{lemma}

\begin{proof}[Proof of Theorem \ref{t.main}]
  {By Propositions \ref{p.facts}-\ref{it.fact2}, \ref{it.fact4} and \ref{it.fact3}},
    \[
  \eqref{UE}+\eqref{C}+\eqref{RVD}\Longrightarrow   \eqref{E}+\eqref{J<}\Longrightarrow  \eqref{EM}+\eqref{J<} .
    \]
    {Together with Corollary \ref{L10}, we obtain,}
    \begin{equation}\label{eq:13-1}
    \eqref{UE}+\eqref{C}+\eqref{RVD}\Longrightarrow \eqref{FK}+ \eqref{lambda}+\eqref{EM}+\eqref{J<}.
    \end{equation}

    On the other hand, using Lemma \ref{L12} {and Proposition \ref{p.facts}-\ref{it.fact5}}, we have
    \begin{equation}
        \eqref{FK}+ \textup{(\hyperlink{L<}{$\operatorname{\lambda_{\leq}}$})}+\eqref{EM}\Longrightarrow \eqref{E} \Longrightarrow \eqref{Gcap}.\label{eq:13-3}
    \end{equation}
  Using {Proposition \ref{p.facts}-\ref{it.fact2}} again, we obtain
 from \eqref{eq:13-3} that
\begin{equation}\label{eq:13-2}
  \eqref{FK}+\textup{(\hyperlink{L<}{$\operatorname{\lambda_{\leq}}$})}+\eqref{EM}+\eqref{J<}\Longrightarrow \eqref{UE}+\eqref{C}+\eqref{RVD}.
\end{equation}
Combining \eqref{eq:13-1} and \eqref{eq:13-2}, we {obtain} \eqref{eq:equ}. 

Clearly, $\eqref{Eb}\Longrightarrow \eqref{EM}$. To prove \eqref{eq:eq+}, it suffices to prove 
\begin{equation*}
  \eqref{FK}+ \textup{(\hyperlink{L<}{$\operatorname{\lambda_{\leq}}$})}+\eqref{EM}\Longrightarrow  \eqref{FK}+ \textup{(\hyperlink{L<}{$\operatorname{\lambda_{\leq}}$})}+\eqref{Eb},
\end{equation*}
which is followed by \eqref{eq:13-3} and Proposition \ref{p.facts}-\ref{it.fact3}.
The proof is complete.
\end{proof}

\section{A note on condition \texorpdfstring{$(\mathrm{E_M})$}{(E\_M)} }\label{s.DisEm}

 A notable property of condition \eqref{EM} in this version is its invariance under quasi-symmetric changes of the metric, {provided that the metric is uniformly perfect (a mild condition)}. Recall that {a} metric space $(M,d)$ is said to be \emph{$K$-uniformly perfect}, if $B(x,r)\setminus B(x,{K^{-1}r})\neq\emptyset$ for any ball $B(x,r)$ such that $M\setminus B(x,r)\neq\emptyset$, and \emph{uniformly perfect} if it is $K$-uniformly perfect for some $K\in(1,\infty)$. It is well-known that the uniformly perfect property is (quantitatively) invariant under quasi-symmetric change of metric \cite[Exercise 11.2]{Hei01}.
\begin{theorem}\label{t.QS}
Condition \eqref{EM} is quasi-symmetric-invariant on uniformly perfect metric spaces.
\end{theorem}
\begin{proof}
Let \(\theta\) and \(d\) be two metrics on \(M\). We introduce the following notation: we write \(\overline R_d\) and \(\overline R_\theta\) for the diameters of \(M\) with respect to \(d\) and \(\theta\), respectively. {For any $(x,r)\in M\times(0,\infty)$}, we define 
\[
{B_d(x,r):=\{y\in M: d(x,y)<r\} \ \text{and}\ B_{\theta}(x,r):=\{y\in M: \theta(x,y)<r\}.}
\]

Suppose that $\theta$ and $d$ are quasi-symmetric, which means that there is a homeomorphism $\eta:[0,\infty)\to[0,\infty)$ such that  \begin{equation}
\label{e.QS}
\frac{\theta(x,y)}{\theta(x,z)}\leq\eta\left(\frac{d(x,y)}{d(x,z)}\right),\ \text{for all }x,y,z\in M\text{ with }x\neq z.
\end{equation}
See \cite[Definition 1.2.9]{MT10}. Tukia--V\"{a}is\"{a}l\"{a}'s theorem \cite[Lemma 1.2.19]{MT10} tells us $\overline{R}_{d}<\infty$ if {and} only if $\overline{R}_{\theta}<\infty$. 

Suppose \eqref{EM} holds with respect to $d$, under which $M$ is $K$-uniformly perfect for some $K\in(1,\infty)$. Let $C>0$, $A>1$ and $\sigma\in(0,1)$ be the constant in \eqref{EM} with respect to $d$ and let $\mathcal{N}\subset M$ be the properly exceptional set in \eqref{EM}. We aim to prove that \eqref{EM} holds under $\theta$ with {the} same properly exceptional set $\mathcal{N}$, that is to prove:
there exists $C_{1}>0, A_{1}>1$ and $\delta\in(0,1)$ such that for any $x\in M$ and $r\in(0,\delta\overline{R}_{\theta})$,
\begin{equation}\label{e.EMtheta}
 C_{1}\sup_{y\in B_{\theta}\left(x,r\right)\setminus\mathcal{N}}\mathbb{E}_{y}\left[\tau_{B_{\theta}\left(x,r\right)}\right]\leq \mathbb{E}_{x}\left[\tau_{B_{\theta}\left(x,{A_1r}\right)}\right].
\end{equation}

By \cite[Lemma 1.2.18]{MT10} (see also \cite[Proposition 3.2]{KM23}), there exists $A_{0}=\eta(A)$ and there exists $t>0$ such that \begin{equation}\label{e.EmQS1}
B_{\theta}(x,r)\subset B_{d}(x,t)\subset B_{d}(x,At)\subset B_{\theta}(x,A_{0}r),
\end{equation}where $t$ can be chosen as (see \cite[Proposition 3.2 and (3.6)]{KM23}),\begin{equation}\label{e.QSt}
t:=A^{-1}\sup\{s\in[0,2 \overline{R}_{d}): B_{d}(x,As)\subset B_{\theta}(x,A_{0}r)\}.
\end{equation} 
By \eqref{e.EmQS1},
\begin{equation}\label{e.EmQS2}
\tau_{B_{\theta}(x,r)}\leq\tau_{B_{d}(x,t)}\leq\tau_{B_{d}(x,At)}\leq\tau_{B_{\theta}(x,A_{0}r)}.
\end{equation}
Thus, if $t\in (0,\sigma\overline{R}_{d})$ then\begin{align}
C\sup_{y\in B_{\theta}\left(x,r\right)\setminus\mathcal{N}}\mathbb{E}_{y}\left[\tau_{B_{\theta}\left(x,r\right)}\right]&\leq C\sup_{y\in B_{d}\left(x,t\right)\setminus\mathcal{N}}\mathbb{E}_{y}\left[\tau_{B_{d}\left(x,t\right)}\right]\quad \text{(by \eqref{e.EmQS1} and \eqref{e.EmQS2})}\\
&\leq \mathbb{E}_{x}\left[\tau_{B_{d}\left(x,At\right)}\right] \quad \text{(by \eqref{EM} under $d$)}\\
&\overset{\eqref{e.EmQS2}}{\leq} \mathbb{E}_{x}\left[\tau_{B_{\theta}\left(x,A_{0}r\right)}\right]\leq\mathbb{E}_{x}\left[\tau_{B_{\theta}\left(x,(1+A_{0})r\right)}\right].\label{e.EmQS3}
\end{align}

Suppose $\overline{R}_{\theta}=\overline{R}_{d}=\infty$. Then we are done by choosing $C_{1}=C$, $A_{1}=1+A_{0}$ and choosing an arbitrary $\delta\in(0,1)$ in \eqref{e.EMtheta}.

Suppose $\overline{R}_{\theta}<\infty$ and $\overline{R}_{d}<\infty$. We first use triangle inequality and see 
\begin{equation}\label{e.QS1}
\begin{aligned}
\frac{1}{2}\inf_{z\in M}\frac{d(x,y)}{d(x,z)}&=\frac{1}{2}\frac{d(x,y)}{\sup_{z\in M}d(x,z)}
\leq \frac{d(x,y)}{\overline{R}_{d}}\\
&{\leq}\frac{d(x,y)}{\sup_{z\in M}d(x,z)}=\inf_{z\in M}\frac{d(x,y)}{d(x,z)}
\end{aligned}
\end{equation}
and similarly, 
\begin{align}\label{e.QS2}
\frac{1}{2}\inf_{z\in M}\frac{\theta(x,y)}{\theta(x,z)}\leq \frac{\theta(x,y)}{\overline{R}_{\theta}}\leq\inf_{z\in M}\frac{\theta(x,y)}{\theta(x,z)}.
\end{align}

By \eqref{e.EmQS3}, in order to prove \eqref{e.EMtheta}, it suffices to choose $\delta\in(0,1)$ such that if $r\in(0,\delta\overline{R}_{\theta})$ then $t\in (0,\sigma\overline{R}_{d})$. Let $\delta\in (0,1)$ to be determined later, and let $r\in(0,\delta\overline{R}_{\theta})$. Note that 
\begin{align}
S_{1}&:=\left\{s\in[0,2 \overline{R}_{d}): B_{d}(x,As)\subset B_{\theta}(x,A_{0}r)\right\}\\
&\overset{\eqref{e.QS1},\eqref{e.QS2}}{\subset} \left\{s\in[0,2 \overline{R}_{d}): \text{ if }\inf_{z\in M}\frac{d(x,y)}{d(x,z)}<\frac{As}{\overline{R}_{d}} \text{ then } \inf_{z\in M}\frac{\theta(x,y)}{\theta(x,z)}< 2A_{0}\delta\right\}\\
&\overset{\eqref{e.QS}}{\subset} \left\{s\in[0,2 \overline{R}_{d}): \text{ if }\inf_{z\in M}\frac{d(x,y)}{d(x,z)}<\frac{As}{\overline{R}_{d}} \text{ then } \inf_{z\in M}\frac{d(x,y)}{d(x,z)}< \frac{1}{\eta^{-1}(2^{-1}A_{0}^{-1}\delta^{-1})}\right\}\\
&\overset{\eqref{e.QS1}}{\subset} \left\{s\in[0,2 \overline{R}_{d}): \text{ if }d(x,y)<\frac{1}{2}As\text{ then } d(x,y)< \frac{\overline{R}_{d}}{\eta^{-1}(2^{-1}A_{0}^{-1}\delta^{-1})}\right\}\\
&=\left\{s\in[0,2 \overline{R}_{d}): B_{d}\left(x,\frac{1}{2}As\right)\subset B_{d}\left(x,\frac{\overline{R}_{d}}{\eta^{-1}(2^{-1}A_{0}^{-1}\delta^{-1})}\right)\right\}=:S_{2},\label{e.QS0+2}
\end{align}
where $\eta^{-1}$ is the inverse function of $\eta$.

By the definition of diameter $\overline{R}_{d}$, 
\begin{equation}\label{e.Cdelta1}
\text{if }\eta^{-1}(2^{-1}A_{0}^{-1}\delta^{-1})\geq4K, \text{ then }M\setminus B_{d}\left(x,\frac{K\cdot\overline{R}_{d}}{\eta^{-1}(2^{-1}A_{0}^{-1}\delta^{-1})}\right)\neq\emptyset,
\end{equation}which implies by the $K$-uniform perfectness of $(M,d)$ that 
\begin{equation}\label{e.QSup}
B_{d}\left(x,\frac{K\cdot\overline{R}_{d}}{\eta^{-1}(2^{-1}A_{0}^{-1}\delta^{-1})}\right)\mathbin{\setminus}\  B_{d}\left(x,\frac{\overline{R}_{d}}{\eta^{-1}(2^{-1}A_{0}^{-1}\delta^{-1})}\right)\neq\emptyset.
\end{equation}
Therefore \begin{equation}\label{e.Cdelta2}
t\overset{\eqref{e.QSt}}{=}A^{-1}\sup S_{1}\overset{\eqref{e.QS0+2}}{\leq}A^{-1}\sup S_{2}\overset{\eqref{e.QSup}}{\leq} \frac{2K\cdot\overline{R}_{d}}{A^{2}\cdot\eta^{-1}(2^{-1}A_{0}^{-1}\delta^{-1})}.
\end{equation}

To make $t\in(0,\sigma\overline{R}_{d})$, by \eqref{e.Cdelta1} and \eqref{e.Cdelta2}, {it} suffices to choose $\delta$ such that
\begin{equation}
\delta\in(0,1),\ \eta^{-1}(2^{-1}A_{0}^{-1}\delta^{-1})\geq4K\text{ and } \frac{2K\cdot\overline{R}_{d}}{A^{2}\cdot\eta^{-1}(2^{-1}A_{0}^{-1}\delta^{-1})}< \sigma\overline{R}_{d},
\end{equation}
which can be done by letting \begin{equation}
\delta=\min\left(\frac{1}{2A_{0}}\eta\left(\frac{2K}{A^{2}\sigma}+4K\right)^{-1},\frac{1}{2} \right).
\end{equation}
The proof is complete.
\end{proof}

However, condition \eqref{Eb} is not stable even under a bi-Lipschitz change of metric, as the following example demonstrates. 

\begin{example}[Failure of bi-Lipschitz invariance of \eqref{Eb}]\label{ex3}

In this example, we show that there exist a uniformly perfect metric measure space $(X,d,\mu)$, a symmetric pure jump regular Dirichlet form $(\mathcal{E},\mathcal{F})$ on $L^2(X,\mu)$, and another metric $\theta$ on $X$ that is bi-Lipschitz equivalent to $d$, such that \eqref{EM} holds with respect to both \(d\) and \(\theta\), whereas \emph{\eqref{Eb} holds with respect to \(d\) but fails with respect to \(\theta\)}; see Proposition \ref{p.exeb} below. In particular, \eqref{EM} does not imply \eqref{Eb}; moreover, \eqref{Eb} is not invariant under bi-Lipschitz changes of metric, and hence is not invariant under quasi-symmetric changes of metric, even on uniformly perfect spaces.

Define $N_n:=n+1$, $\delta_n:=2^{-n-4}$, $n\in\mathbb{N}$. Then \(N_n \uparrow \infty\) and \(\delta_n \downarrow 0\). For each \(n\ge 1\), let $\Omega_n$ be a set with $\#\Omega_{n}=4$, say $ \Omega_n:=\{o_n,s_n,a_n,b_n\}$.
Let $Y:=\bigsqcup_{n\in\bN} \Omega_n$ and $X:=Y\times \mathbb R$.

Define two maps $h,g:Y\to \mathbb{R}$ by setting, for each $n\in\mathbb{N}$, \begin{align}
	   h(o_n)&:=-2\delta_n,\quad h(s_n):=0,\quad h(a_n):=\delta_n,\quad h(b_n):=3\delta_n,\\
	   g(o_n)&:=-2\delta_n,\quad g(s_n):=0,\quad g(a_n):=2\delta_n ,\quad g(b_n):=\delta_n.
\end{align}
Define two functions $d_{Y}, \theta_{Y}:Y\times Y\to(0,\infty)$ by \begin{equation}\label{e.dY}
	d_{Y}(x,y):=\begin{cases}
	\abs{h(x)-h(y)},\ & \text{ if }(x,y)\in\Omega_{n}\times\Omega_{n},\ n\in\bN.\\
	|n-m|+1,&  \text{ if }(x,y)\in\Omega_{n}\times\Omega_{m},\ n\neq m,
\end{cases}
\end{equation}
and \begin{equation}\label{e.dT}
	\theta_{Y}(x,y):=\begin{cases}
	\abs{g(x)-g(y)},\ & \text{ if }(x,y)\in\Omega_{n}\times\Omega_{n},\ n\in\bN.\\
	|n-m|+1,&  \text{ if }(x,y)\in\Omega_{n}\times\Omega_{m},\ n\neq m.
	\end{cases}
\end{equation}
Define $d,\theta:X\times X\to \bR$ by 
\begin{equation}\label{e.ef1}
	\begin{cases}
	d(x,y):=d_Y(p,q)\vee |u-v|\\
	\theta(x,y):=\theta_Y(p,q)\vee |u-v|
\end{cases} \text{ for every } x=(p,u),\ y=(q,v).
\end{equation}
\begin{lemma}\label{l.bil}
	The functions $d$ and $\theta$ defined in \eqref{e.ef1} are bi-Lipschitz equivalent metrics on $X$, and the metric spaces $(X,d)$ and $(X,\theta)$ are locally compact, separable, and uniformly perfect.
\end{lemma}
\begin{proof}
To prove that $d$ is a metric, it suffices to verify the triangle inequality. For this, we first verify the triangle inequality for \(d_Y\). Inside each block $\Omega_n$, the distance \(d_Y\) is induced by the Euclidean metric on $\bR$. Hence the triangle inequality holds whenever the three points belong to the same block. Let \(x,y,z\in Y\). If two points are in the same block and the third point is in a different block, we may assume \(x,y\in \Omega_n\) and \(z\in \Omega_m\), \(m\neq n\); then, by \eqref{e.dY},
\begin{equation*}
 d_Y(x,z)=d_Y(y,z)=|n-m|+1>d_Y(x,y),
\end{equation*}
and hence
\[
    d_Y(x,z)\vee d_Y(y,z) \le d_Y(x,y)+d_Y(x,z)\wedge d_Y(y,z).
\]
Finally, suppose $x\in \Omega_n$, $y\in \Omega_k$, $z\in \Omega_m$, with $n,k,m$ pairwise distinct. Then $d_Y(x,z)=|n-m|+1$, while
\[
    d_Y(x,y)+d_Y(y,z) = (|n-k|+1)+(|k-m|+1)\geq |n-m|+2.
\]
Thus $d_Y(x,z)\le d_Y(x,y)+d_Y(y,z)$. Hence \(d_Y\) is a metric.
 
Given three points $x=(p,u)$, $y=(q,v)$, and $z=(r,w)$ in $X=Y\times \bR$, since \(d_Y\) is a metric, we have
\[
\begin{aligned}
    d(x,z)
    &=
    d_Y(p,r)\vee |u-w|   \le
    \bigl(d_Y(p,q)+d_Y(q,r)\bigr)
    \vee
    \bigl(|u-v|+|v-w|\bigr) \\
    &\le
    \bigl(d_Y(p,q)\vee |u-v|\bigr)
    +
    \bigl(d_Y(q,r)\vee |v-w|\bigr) 
    =
    d(x,y)+d(y,z).
\end{aligned}
\]
Thus the triangle inequality for \(d\) holds and therefore $d$ is a metric on $X$. The same proof shows that \(\theta\) is also a metric on $X$, and $(X,\theta)$ is also locally compact and separable. By \eqref{e.dY}, \eqref{e.dT}, and \eqref{e.ef1}, we know that $3^{-1}d_Y\le \theta_Y\le 3d_Y$, so $d$ and $\theta$ are bi-Lipschitz equivalent. 

The local compactness and separability of $(X,d)$ follow from the corresponding properties of $(Y,d_{Y})$ and $\bR$. To show the uniform perfectness of $(X,d)$, fix \(x=(p,u)\in X\) and \(r\in(0,\infty)\). The point \((p,u+r/2)\) satisfies $d(x,(p,u+r/2))=r/2$.
Hence $B_d(x,r)\setminus B_d(x,r/3)\neq \emptyset$. Since $d$ and $\theta$ are bi-Lipschitz equivalent, $(X,\theta)$ is also locally compact, separable, and uniformly perfect.
\end{proof}

\begin{lemma}
	Let $\#$ be the counting measure on $Y$ and $\rL$ be the Lebesgue measure on $\bR$. Let $\mu:=\#\otimes\rL$ be the product measure on $X=Y\times \bR$. Then the metric measure spaces $(X,d,\mu)$ and $(X,\theta,\mu)$ are volume doubling \eqref{VD} and reverse volume doubling \eqref{RVD}.
\end{lemma}

\begin{proof}
Since $B_d((p,u),r)=B_{d_Y}(p,r)\times (u-r,u+r)$, we have
\begin{equation}\label{e.vd}
 \frac{\mu(B_d((p,u),2r))}{\mu(B_d((p,u),r))}=2 \frac{\#B_{d_Y}(p,2r)}{\#B_{d_Y}(p,r)}.
\end{equation}
Therefore, to show that $\mu$ satisfies \eqref{VD} on $(X,d)$, it suffices to prove that $\#$ satisfies \eqref{VD} on $(Y,d_{Y})$.
\begin{equation}\label{e.nu}
   \#B_{d_Y}(p,2r)
    \le
    8\#B_{d_Y}(p,r),\ \text{for all } (p,r)\in Y\times(0,\infty).
\end{equation}
Indeed, each block $\Omega_n=\{o_n,s_n,a_n,b_n\}$ contains exactly four points. If $0<r\leq 1$, then $B_{d_Y}(p,2r)$ is contained in the same block as $p$, and hence
$\#(B_{d_Y}(p,2r))\le 4$. Since $\#(B_{d_Y}(p,r))\ge 1$, we obtain \eqref{e.nu} for $(p,r)\in Y\times(0,1]$. It remains to consider $(p,r)\in Y\times(1,\infty)$. Suppose $p\in \Omega_n$. By \eqref{e.dY}, $B_{d_Y}(p,r)$ can meet only those blocks \(\Omega_m\) for which
$|n-m|+1<r$. Hence the number of blocks intersecting \(B_{d_Y}(p,r)\) is 
$2\lceil r-1\rceil -1$. Since each block has four points, we have $\#B_{d_Y}(p,r)=  4(2\lceil r-1\rceil -1)=8\lceil r\rceil-12\geq 2\lceil r\rceil$, because $\lceil r\rceil\geq 2$. On the other hand, $ \#B_{d_Y}(p,2r)=8\lceil 2r\rceil-12\leq 16\lceil r\rceil$, which gives \eqref{e.nu} for $(p,r)\in Y\times(1,\infty)$. Thus, $(X,d,\mu)$ is \eqref{VD}. The reverse volume doubling property \eqref{RVD} follows from the uniform perfectness in Lemma \ref{l.bil} and \cite[Exercise 13.1]{Hei01}. The same argument applies to $(X,\theta,\mu)$.
\end{proof}
Define a function $J:(X\times X)\setminus \diag\to [0,\infty)$ by  \begin{equation}\label{e.efg}
	J((p,u),(q,v)):=\begin{dcases}
		1,\ &\text{ if }(p,q)\in\{(a_{n},s_{n}),(s_{n},a_{n})\}\text{ for some }n\in\bN, \text{ and if }u=v,\\
		N_{n}^{-1} ,\ &\text{ if }(p,q)\in\{(a_{n},b_{n}),(b_{n},a_{n})\}\text{ for some }n\in\bN, \text{ and if }u=v,\\
		N_{n}^{-1},\ &\text{ if }(p,q)\in\{(o_{n},s_{n}),(s_{n},o_{n})\}\text{ for some }n\in\bN, \text{ and if }u=v,\\
		0,\ &\text{ otherwise}.
	\end{dcases}
\end{equation}
\begin{proposition}\label{p.exeb}
	Let $J:(X\times X)\setminus \diag\to [0,\infty)$ be given by \eqref{e.efg}. \begin{enumerate}[label=\textup{({\arabic*})},align=right,leftmargin=*,topsep=5pt,parsep=0pt,itemsep=2pt]
	\item\label{it.exeb1}  Define a symmetric bilinear form \((\mathcal E,\mathcal F)\) by
  \begin{equation}\label{e.eDF}
\begin{aligned}
 \mathcal E(f,g)
    &:=\frac{1}{2}\int_{(X\times X)\setminus \diag}(f(x)-f(y))(g(x)-g(y))J(x,y)\mu(\dif x)\mu(\dif y)\\
 \mathcal{F} &:=\text{the closure of }\left\{f\in C_{c}(X):\mathcal{E}(f,f)<\infty \right\}\text{ under the norm }\sqrt{\mathcal{E}_{1}}.
\end{aligned}
\end{equation}
Then $(\sE,\sF)$ is a regular symmetric Dirichlet form on $L^{2}(X,\mu)$.
\item\label{it.exeb3} With respect to $d$, the conditions \eqref{EM} and \eqref{Eb} hold for $(\sE,\sF)$. With respect to $\theta$, the condition \eqref{EM} holds while \eqref{Eb} fails for $(\sE,\sF)$.
\item\label{it.exeb4} For any scale function $W$, the conditions $\textup{(\hyperlink{E<}{$\operatorname{\mathrm{E}_{\leq}}$})}$, \eqref{FK}, and \textup{(\hypertarget{L>}{}\hyperlink{L>}{$\operatorname{\lambda_{\geq}}$})} all fail with respect to both $d$ and $\theta$. However, \textup{(\hypertarget{L<}{}\hyperlink{L<}{$\operatorname{\lambda_{\leq}}$})} and \eqref{J<} hold with $W(x,r)\equiv r$.
\end{enumerate}
\end{proposition}
\begin{proof}\begin{enumerate}[label=\textup{({\arabic*})},align=right,leftmargin=*,topsep=5pt,parsep=0pt,itemsep=2pt]
	\item[\ref{it.exeb1}]  It suffices to prove regularity, as the remaining assertions are evident. Clearly, the space $\mathcal{F}\cap C_{c}(X)$ is dense in 
$\mathcal{F}$ under the $\sqrt{\mathcal{E}_{1}}$-norm by definition. We will show that $C_{c}(X)\subset \sF$ so that $\mathcal{F }\cap C_{c}(X)=C_{c}(X)$ is dense in $(C_{c}(X),\norm{\cdot}_{\sup})$. Let \(f\in C_c(X)\). A direct calculation gives
\begin{align}
	 \mathcal E(f,f)&=\sum_{n\in\bN}\int_{\mathbb R}
   \Bigg( (f(a_n,u)-f(s_n,u))^2
    +\frac1{N_n}(f(a_n,u)-f(b_n,u))^2 \\
&\qquad \qquad\qquad+\frac1{N_n}(f(s_n,u)-f(o_n,u))^2 \Bigg) \dif u    \label{e.exDF}
\end{align}
Since $\mathrm{supp}(f)$ meets only finitely many blocks \(\Omega_n\) and is bounded, the energy \(\mathcal E(f,f)\) is a finite sum of finite integrals.
Consequently, $\mathcal{E}(f,f)<\infty$ and $f\in\sF$.
\item[\ref{it.exeb3}]  Let $(\{Z_{t}\}_{t\in[0,\infty)},\{\bP_{x}\}_{x\in X})$ be a $\mu$-symmetric Hunt process on $X$ corresponding to the regular Dirichlet form $(\sE,\sF)$ on $L^{2}(X,\mu)$.  We first recall that, by \cite[Theorem 3.3.3 \text{with $A=D^c$}]{Nor97}, \begin{equation}
\bE_{x}[\tau_{\{x\}}]=Q(x)^{-1},\ \text{for all } x\in X,
\end{equation} and for any non-empty subset $D$, 
\begin{equation}\label{e.jump}
\bE_{x}[\tau_{D}] = \frac1{Q(x)}  +  \sum_{y\in D\setminus\{x\}} \frac{J(x,y)}{Q(x)}\bE_{y}[\tau_{D}],\ \text{for all }x\in D,
\end{equation}
where \begin{equation}
	Q(x):=\sum_{y\in X\setminus \{x\}}J(x,y),\ x\in X.
\end{equation}

Fix $n\in\bN$. In the \(d_Y\)-metric, the possible proper \(d_Y\)-balls are
\begin{align}
	   &\{o_{n}\},\ \{s_{n}\},\ \{a_{n}\},\ \{b_{n}\},\ \{o_{n},s_{n}\},\ \{s_{n},a_{n}\},\ \\
	  & \{a_{n},b_{n}\}, \ \{o_{n},s_{n},a_{n}\},\ \{s_{n},a_{n},b_{n}\},\  \bigsqcup_{k\in[i,j]\cap\bN}\Omega_{k} \text{ with } n\in [i,j]\cap\bN.
\end{align}
(The set \(\{a_n,b_n\}\) may occur as a \(d_Y\)-ball centered at \(b_n\), but it
cannot occur as a \(d_Y\)-ball centered at \(a_n\), because $d_Y(a_n,s_n)=\delta_n<2\delta_n=d_Y(a_n,b_n)$.) Solving \eqref{e.jump} gives the following values.
\begin{enumerate}[label=\textup{(\roman*)},align=right,leftmargin=*,topsep=5pt,parsep=0pt,itemsep=2pt]
\item\label{it.p1} For singletons, we have $ \bE_{a_{n}}[{\tau_{\{a_{n}\}}}]=\bE_{s_{n}}[{\tau_{\{s_{n}\}}}]={N_{n}}({N_{n}+1})^{-1}$ and $\bE_{b_{n}}[{\tau_{\{b_{n}\}}}]=\bE_{o_{n}}[{\tau_{\{o_{n}\}}}]=N_{n}$.
\item\label{it.p2} For two-point sets, $\bE_{a_{n}}[{\tau_{\{s_{n},a_{n}\}}}]=\bE_{s_{n}}[{\tau_{\{s_{n},a_{n}\}}}]=N_{n}$, $\bE_{a_{n}}[{\tau_{\{a_{n},b_{n}\}}}]=2$, $\bE_{b_{n}}[{\tau_{\{a_{n},b_{n}\}}}]=N_{n}+2$, and $\bE_{s_{n}}[{\tau_{\{o_{n},s_{n}\}}}]=2$, $\bE_{o_{n}}[{\tau_{\{o_{n},s_{n}\}}}]=N_{n}+2$.
\item\label{it.p3} For three-point sets,
\[
    \bE_{o_{n}}[\tau_{\{o_{n},s_{n},a_{n}\}}]=4N_n+2,
    \quad
    \bE_{s_{n}}[\tau_{\{o_{n},s_{n},a_{n}\}}]=3N_n+2,
    \quad
    \bE_{a_{n}}[\tau_{\{o_{n},s_{n},a_{n}\}}]=3N_n,
\]
\[
    \bE_{s_{n}}[\tau_{\{s_{n},a_{n},b_{n}\}}]=3N_n,
    \quad
    \bE_{a_{n}}[\tau_{\{s_{n},a_{n},b_{n}\}}]=3N_n+2,
    \quad
    \bE_{b_{n}}[\tau_{\{s_{n},a_{n},b_{n}\}}]=4N_n+2.
\]
\item\label{it.p4}  For all $i\leq j$ such that $n\in[i,j]\cap \bN$ and for all $x\in \Omega_{n}$, we always have $\bE_x[\tau_{\Omega_{n}}]=\infty$, because \(\Omega_{n}\) is an invariant class for the process $Z$.
\end{enumerate}
Combining the above discussion, we always have 
\[
    \sup_{z\in B_{d_Y}(x,r)} \bE_z[\tau_{B_{d_Y}(x,r)}] \le 2 \bE_x[\tau_{B_{d_Y}(x,r)}].
\]
Since the process $Z$ does not move in the \(\mathbb R\)-coordinate, the same estimate holds for \(d\)-balls in \(X\). Therefore, $\sup_{z\in B_d(x,r)}
    \mathbb E_z\tau_{B_d(x,r)}
    \le 2\mathbb E_x\tau_{B_d(x,r)}$.
That is, \eqref{Eb} holds with respect to \(d\). By Proposition \ref{p.facts}-\ref{it.fact3}, \eqref{EM} also holds with respect to \(d\). By Theorem \ref{t.QS} and Lemma \ref{l.bil}, \eqref{EM} also holds with respect to \(\theta\).
We then show that \eqref{Eb} fails with respect to \(\theta\). Note that $\theta_Y(a_n,b_n)=\delta_n$ and
$ \theta_Y(a_n,s_n)=2\delta_n$. Then $B_{\theta_Y}\left(a_n,3\delta_n/2\right)=\{a_n,b_n\}$.
For this two-point set, the computation above gives $\mathbb E_{a_n}[\tau_{\{a_n,b_n\}}]=2$,
whereas $ \mathbb E_{b_n}[\tau_{\{a_n,b_n\}}]=N_n+2$. Hence
\[
\frac{
    \sup_{z\in B_{\theta_Y}\left(a_n,\frac32\delta_n\right)}
    \mathbb E_z\left[\tau_{B_{\theta_Y}\left(a_n,\frac32\delta_n\right)}\right]
}{   \mathbb E_{a_n}\left[\tau_{B_{\theta_Y}\left(a_n,\frac32\delta_n\right)}\right]
}=\frac{N_n+2}{2}\uparrow \infty\ \text{ as }n\uparrow \infty.
\]
The same conclusion holds in \(X=Y\times\mathbb R\), for instance at the
points \((a_n,0)\), because the process keeps the second coordinate fixed.
Thus there is no uniform constant \(C\) such that $\sup_{z\in B_\theta(x,r)}
    \mathbb E_z\tau_{B_\theta(x,r)}
    \le
    C\,\mathbb E_x\tau_{B_\theta(x,r)}$ for all $(x,r)\in X\times (0,\infty)$. Therefore \eqref{Eb} fails with respect to \(\theta\).
\item[\ref{it.exeb4}] Since $\textup{(\hyperlink{E<}{$\operatorname{\mathrm{E}_{\leq}}$})}$ fails by \ref{it.p4}, the Faber--Krahn inequality \eqref{FK} also fails by Proposition \ref{p.facts}-\ref{it.fact5}. Now we show that (\hypertarget{L>}{}\hyperlink{L>}{$\operatorname{\lambda_{\geq}}$}) fails with respect to $d$.
Fix $n\ge1$ and $u_0\in\mathbb R$. Let $x_n:=(a_n,u_0)$ and $r_n:=4\delta_n$; then $B_d(x_n,r_n)=\Omega_n\times (u_0-r_n,u_0+r_n)$. Choose a non-zero function $\varphi\in C_c((u_0-r_n,u_0+r_n))$ and define $f(p,u):=\one_{\Omega_{n}}(p)\varphi(u)$ for $(p,u)\in Y\times \bR=X$. Then $f$ is supported in \(B_d(x_n,r_n)\), and $\|f\|_{L^2(B_d(x_n,r_n),\mu)}^2>0$.
Since $f$ is constant on the block \(\Omega_n\) for each fixed
\(u\) and there are no jumps between different blocks, we have from \eqref{e.eDF} that $\mathcal E(f,f)=0$. It follows that $\lambda\bigl(B_d(x_n,r_n)\bigr)=0$ by definition. The same argument applies to \(\theta\). Thus (\hypertarget{L>}{}\hyperlink{L>}{$\operatorname{\lambda_{\geq}}$}) fails for both $d$ and $\theta$.\\
Next, we show that (\hypertarget{L<}{}\hyperlink{L<}{$\operatorname{\lambda_{\leq}}$}) holds for both $d$ and $\theta$. Let $x=(p,u)\in Y\times \mathbb{R}$ and $r>0$. Then $p\in \Omega_n$ for some $n\in \mathbb{N}$. If $r> 1$, then $B_{d_Y}(p,r)$ contains the block \(\Omega_n\) in $Y$, since $\mathrm{diam}(\Omega_n)<1$. From the above argument, we have $\lambda\bigl(B_d(x,r)\bigr)=0$ in this case; thus (\hypertarget{L<}{}\hyperlink{L<}{$\operatorname{\lambda_{\leq}}$}) holds. Now, assume that $r\leq 1$. We choose the test function $f(p,\cdot)\in L^2(\mathbb{R})$ and $f(q,\cdot)=0$ for any $q\in Y\setminus \{p\}$. Since $p$ is connected to at most two points in $\Omega_n$ with rate $1$ or ${N_n}^{-1}$, \eqref{e.exDF} gives
\begin{equation*}
    \mathcal E(f,f)\le\sum_{n\in\bN}\int_{\mathbb R}
   \Bigg( f(p,u)^2
    +\frac1{N_n}f(p,u)^2 \Bigg) \dif u = (1+{N_n}^{-1})\|f\|_{L^2(B_d(x,r))}^2,
\end{equation*}
which implies $\lambda\bigl(B_d(x,r)\bigr)\leq (1+{N_n}^{-1})\leq2 \leq 2r^{-1}$. Similarly,
$\lambda\bigl(B_\theta(x,r)\bigr)\leq 2r^{-1}$.
Therefore (\hypertarget{L<}{}\hyperlink{L<}{$\operatorname{\lambda_{\leq}}$}) holds with $W(x,r)=r$ for both $d$ and $\theta$. The upper bounds \eqref{J<} for $J$ under $d$ and $\theta$ follow from a direct calculation using \eqref{e.efg}.
\end{enumerate}
\end{proof}
\end{example}

Formally, condition \eqref{EM} is not easy to verify directly; however, this challenge can be circumvented by leveraging Theorem \ref{t.main} for its validation.
\begin{example}\label{ex.ex1}
Let $(\mathbb{R}^n,\lvert\cdot\rvert,\dif x)$ be the $n$-dimensional Euclidean space with Euclidean distance and Lebesgue measure. Let $(\mathcal{E},\mathcal{F})$ be a Dirichlet form in $L^{2}(%
\mathbb{R}^n,\dif x)$ given as in \cite[formula (1.5)]{CK10} with $A(x)$ being the identity matrix, that is, 
\begin{equation}\label{e.EXE}
\mathcal{E}(u,v)=\frac{1}{2}\int_{\mathbb{R}^n}\nabla u(x)\cdot \nabla
v(x)\dif x+\int_{\mathbb{R}^n}\int_{\mathbb{R}
^{n}}(u(x)-u(y))(v(x)-v(y))J(x,y)\dif x\dif y,
\end{equation}
where the jump kernel is given by
\begin{equation}
J(x,y)=\frac{1}{|x-y|^{n+2s}} \label{jump}
\end{equation}%
for some $0<s<1$, and the domain $\mathcal{F}$ given by 
\begin{equation}\label{e.EXF}
\mathcal{F}=\overline{C_{c}^{1}(\mathbb{R}^n)}^{\mathcal{E}_{1}},
\end{equation}
where $C_{c}^{1}(\mathbb{R}^n)$ is the space of all $C^{1}$-functions {with compact support on $\mathbb{R}^n$}. As $C_{c}^{1}(\mathbb{R}%
^{d})\subset \mathcal{F}\cap C_{c}(\mathbb{R}^n)$ is both dense in $%
\mathcal{F}$ under the norm of $\sqrt{\mathcal{E}(\cdot ,\cdot )+\lVert\cdot
\rVert_{L^{2}(\mathbb{R}^n,\dif x)}^{2}}$ and dense in $C_{c}(\mathbb{R}^n)$ under the supremum norm,
the bilinear form $(\mathcal{E},\mathcal{F})$ is a regular Dirichlet form in $L^{2}(%
\mathbb{R}^n,\dif x)$.

In this case, conditions \eqref{VD}, \eqref{RVD} are trivially satisfied with $
V(x,r)=cr^{n}$, where $c$ is the volume of the unit ball in $\mathbb{R}^n$. By \eqref{jump}, all the hypotheses in \cite[Theorem 1.4]{CK10} are satisfied. Therefore, the heat kernel of the Dirichlet form $(\mathcal{E},\mathcal{F})$ exists, and satisfies the following two-sided estimate: there are constants $c_{j}$, $j\in\{1,2,3,4\}$ such that for every $t>0$ and $x,y\in \mathbb{R}^n$, 
\begin{align}
&\phantom{\ \leq} c_{1}\left( t^{-n/2}\wedge t^{-n/2s}\right) \wedge \left(
p^{(c)}(t,c_{2}|x-y|)+p^{(j)}(t,|x-y|)\right)\\
& \leq p_{t}(x,y)  \\
& \leq c_{3}\left( t^{-n/2}\wedge t^{-n/2s}\right) \wedge \left(
p^{(c)}(t,c_{4}|x-y|)+p^{(j)}(t,|x-y|)\right) ,  \label{11}
\end{align}%
where \[p^{(c)}(t,r):=t^{-n/2}\exp(-\frac{r^{2}}{t})\text{ and }p^{(j)}(t,r):=t^{-n/2s}\wedge \frac{t}{r^{n+2s}}.\]

Define the scale function $W$ by 
\begin{equation*}
W(x,r):=r^{2}\wedge r^{2s}\ \text{for $x\in \mathbb{R}^n$ and $r>0$}.
\end{equation*}%
We will show that \eqref{UE} holds with $W$, that is
\begin{equation}
p_{t}(x,y)\leq C\left( ( t^{-n/2}\wedge t^{-n/2s})\wedge \frac{t}{|x-y|^n (|x-y|^2\wedge |x-y|^{2s})}\right) .\label{UE2}
\end{equation}
Indeed, using the elementary inequalities: 
\begin{equation*}
\frac{1}{a\wedge b}\geq \frac{1}{2}\left( \frac{1}{a}+\frac{1}{b}\right) 
\text{ \ and \ }\frac{1}{a}\geq \exp \left( -e^{-1}a\right)\ \text{ for any $a,b>0$,}
\end{equation*}%
we know that%
\begin{align}
\frac{t}{r^{n}(r^{2}\wedge r^{2s})} &\geq \frac{1}{2}\left( \frac{t}{r^{n+2}%
}+\frac{t}{r^{n+2s}}\right) =\frac{1}{2}\left( t^{-n/2}\left( \frac{r^{2}}{t}%
\right) ^{-n/2-1}+\frac{t}{r^{n+2s}}\right)  \\
&\geq \frac{1}{2}\left( t^{-n/2}\exp \left( -\left( \frac{n}{2}+1\right)
e^{-1}\frac{r^{2}}{t}\right) +t^{-n/(2s)}\wedge \frac{t}{r^{n+2s}}\right) .
\end{align}%
From this, it follows that%
\begin{align}
&\phantom{\ \leq}(t^{-n/2}\wedge t^{-n/2s})\wedge \frac{t}{r^{n}(r^{2}\wedge r^{2s})} \\
&\geq \frac{1}{2}\left( (t^{-n/2}\wedge t^{-n/2s})\wedge \left(
p^{(c)}(t,c_{4}r)+p^{(j)}(t,r)\right) \right) , \label{10}
\end{align}
with $c_{4}=\sqrt{\left( n/2+1\right) e^{-1}}$. Letting $r=|x-y|$ in \eqref{10} and combining \eqref{11}, we obtain \eqref{UE2}.

Since $(\mathcal{E},\mathcal{F})$ is conservative by 
\cite[Theorem 2.2]{CK10}, we can now apply Theorem \ref{t.main} and see that conditions \eqref{FK}, \textup{(\hyperlink{L<}{$\operatorname{\lambda_{\leq}}$})} and \eqref{EM} hold. 

However, the heat kernel for $(\mathcal{E},\mathcal{F})$ cannot have a sub-Gaussian heat kernel bound like \eqref{UEloc}. Since for $t\in (0,1)$ and $x,y$ satisfying $|x-y|>1$, we see from the lower bound of \eqref{11} that
\begin{equation*}
  p_t(x,y)\geq \frac{Ct}{|x-y|^{n+2s}},
\end{equation*}
which implies an exponential decay of heat kernel cannot hold for large $|x-y|$.
\end{example}

\begin{example}[\text{\cite[Example 2.1]{GHH24a}}]
On $(\mathbb{R}^n,\lvert\cdot\rvert,\dif x)$, for $0<\varepsilon<\beta<2$, set  
\begin{equation*}
W(x,r) = \left(\frac{|x|+r}{r}\right)^\varepsilon r^\beta,\quad x \in 
\mathbb{R}^n,~r >0.
\end{equation*}
It is easy to prove that $W(x,\cdot)$ is strictly increasing {and} satisfies \eqref{eq:vol_0} with $\beta_1=\beta-\epsilon$,
$\beta_2=\beta$ and $C=2^\epsilon$. In this case, conditions \eqref{VD} and \eqref{RVD} trivially hold. Consider {a symmetric measurable function} $J$ satisfying  
\begin{equation*}
{\frac{C^{-1}}{|x-y|^n W(x,|x-y|)} \leq J(x,y) \leq \frac{C}{|x-y|^n W(x,|x-y|)},}
\end{equation*}
{for some constant $C>1$,} and denote the Dirichlet form associated with the above jump kernel $J$ as in \eqref{EJ} by $\mathcal{E}^{(W)}$.
By \cite[Lemma 8.1 and Remark 4.16-(2)]{GHH23}, we see that condition \eqref{FK} holds. Also, by \cite[Lemma 3.2]{CGHL25} {and \cite[Lemma 7.15]{GHH23}}, we obtain condition \eqref{Gcap}. Hence by {Proposition \ref{p.facts}-\ref{it.fact2}}, we obtain condition \eqref{UE}, thus showing conditions \textup{(\hyperlink{L<}{$\operatorname{\lambda_{\leq}}$})}, \eqref{EM} by applying Theorem \ref{t.main}.
\end{example}

\noindent \textbf{Acknowledgments.} 
The authors are grateful to Jiaxin Hu for suggesting the problem tackled in this paper. We also thank Eryan Hu for helpful discussions on this paper. The authors are very grateful to the anonymous referees for their careful reading of this manuscript and many helpful suggestions. Aobo Chen is supported by the research grant (VIL73729) from Villum Fonden. Zhenyu Yu is supported by the Natural Science Foundation of Hunan Province, China (No.~2025JJ60039) and Innovation Research Foundation of National University of Defense Technology (No.~ZK25-05).

\noindent \textbf{Statements:} All authors contributed equally to this work and approved the final manuscript; no datasets were generated or analyzed during the current study; the authors declare no conflict of interest; ethical approval, patient consent, permission to reproduce material from other sources, and clinical trial number are not applicable.

\vspace{10pt}
\noindent Aobo Chen

\vspace{3pt}
\noindent Department of Mathematics, Aarhus University, 8000 Aarhus C, Denmark
\vspace{3pt}

\noindent \texttt{aobochen.math@hotmail.com}, \texttt{aobochen@math.au.dk}

\vspace{10pt}
\noindent Zhenyu Yu

\vspace{3pt}
\noindent College of Science, National University of Defense Technology, Changsha 410073, China

\vspace{3pt}
\noindent \texttt{yuzy23@nudt.edu.cn}

\end{document}